\pdfoutput=1
\newif\ifpersonal
\personaltrue % comment to remove personal notes
\documentclass[12pt,a4paper,dvipsnames,x11names]{amsart} %reqno places equation numbers on the right
\usepackage{amsmath,amsthm,amssymb,mathrsfs,mathtools,bm,eucal,tensor} % math related
\usepackage{stmaryrd}
\usepackage[all,cmtip]{xy}
\usepackage{dsfont}
\usepackage{microtype,lmodern} % latex technical issues
\usepackage[utf8]{inputenc} % input encoding
\usepackage[T1]{fontenc} % font encoding
\usepackage{enumerate,comment,braket,xspace,tikz-cd,csquotes} % utilities
\usepackage[centering,vscale=0.75,hscale=0.75]{geometry}
\usepackage[hidelinks,colorlinks=true,linkcolor=Red4,citecolor=RoyalBlue]{hyperref}
\usepackage[capitalize]{cleveref}

\usepackage{xcolor}

\usepackage{mathpazo}
\usepackage{euler}

\newtheorem{theorem}[subsection]{Theorem}
\newtheorem{proposition}[subsection]{Proposition}
\newtheorem{corollary}[subsection]{Corollary}
\newtheorem{lemma}[subsection]{Lemma}

\theoremstyle{definition}

\newtheorem{remark}[subsection]{Remark}

\numberwithin{equation}{subsection}

\usepackage{etoolbox}

\patchcmd{\section}{\scshape}{\bfseries}{}{}
\makeatletter
\renewcommand{\@secnumfont}{\bfseries}
\makeatother

\newcommand{\spec}{\mathrm{Spec}}
\newcommand{\sO}{\mathscr{O}}

\newcommand{\sF}{\mathscr{F}}
\newcommand{\ra}{\rightarrow}

\newcommand{\bA}{\mathbb A}

\newcommand{\rC}{\mathrm{C}}
\newcommand{\rLC}{\mathrm{LC}}
\newcommand{\rlc}{\mathrm{lc}}
\newcommand{\rc}{\mathrm{c}}
\newcommand{\dt}{\mathrm{dt}}
\newcommand{\sw}{\mathrm{sw}}
\newcommand{\rGC}{\mathrm{GC}}
\newcommand{\rGLC}{\mathrm{GLC}}
\newcommand{\Spec}{\mathrm{Spec}}
\newcommand{\bT}{\mathbb T}
\newcommand{\sK}{\mathscr K}
\newcommand{\ol}{\overline}
\newcommand{\rk}{\mathrm{rk}}

\usepackage[english]{babel}

\usepackage[english]{babel}

\begin{document}

\title{semi-continuity for conductor divisors of \'etale sheaves}
%\linenumbers
%\begin{document}\Large

%%%%%%%%%%%%% first author
\author{Haoyu Hu}
\address{School of Mathematics, Nanjing University, Hankou Road 22, Nanjing, China}
\email{huhaoyu@nju.edu.cn, huhaoyu1987@gmail.com}

%%%%%%%%%%%%%%%% second author
\author{Jean-Baptiste Teyssier}
\address{Institut de Math\'ematiques de Jussieu, 4 place Jussieu, Paris, France}
%\curraddr{Department of Mathematics and Statistics,
%Case Western Reserve University, Cleveland, Ohio 43403}
\email{jean-baptiste.teyssier@imj-prg.fr}
%\curraddr{Department of Mathematics and Statistics,
%Case Western Reserve University, Cleveland, Ohio 43403}

%\thanks{}

%\thanks{}

%    General info

%\date{August 15, 2011 and, in revised form, June 22, 2001.}



\begin{abstract}
In this article, we prove a semi-continuity property for both conductor divisors and logarithmic conductor divisors for \'etale sheaves on higher relative dimensions in a geometric situation. It generalizes a semi-continuity result for conductors of \'etale sheaves on relative curves to higher relative dimensions, and it can be considered as a higher dimensional $\ell$-adic analogy of Andr\'e's result on the semi-continuity of Poincar\'e-Katz ranks of meromorphic connections on smooth relative curves.

\end{abstract}
\maketitle

\tableofcontents

\section{Introduction}
\subsection{}\label{delignelaumon}
Let $S$ be an excellent noetherian scheme, $f:X\rightarrow S$ a separated and smooth morphism of relative dimension $1$, $D$ a closed subset of $X$ which is finite and flat over $S$, $U$ the complement of $D$ in $X$ and $j:U\rightarrow X$ the canonical injection. Let $\ell$ be a prime number invertible in $S$ and $\Lambda$ a finite field of characteristic $\ell$. Let $\mathscr F$ be a locally constant and constructible sheaf of $\Lambda$-modules on $U$ of constant rank. For any point $ s\in S$, we denote by $\bar s\rightarrow S$ an algebraic geometric point above $s$ and by $X_{\bar s}$ and $D_{\bar s}$ the fibers of $f:X\rightarrow S$ and $f|_D:D\rightarrow S$ at $\bar s$, respectively. For each point $x\in D_{\bar s}$, we define
\begin{equation}\label{dimtotcurve}
\mathrm{dt}_x(j_!\mathscr F|_{X_{\bar s}})=\mathrm{sw}_x(j_!\mathscr F|_{X_{\bar s}})+\mathrm{rank}(\mathscr F),
\end{equation}
where $\mathrm{sw}_x(j_!\mathscr F|_{X_{\bar s}})$ denotes the classical Swan conductor of the sheaf $j_!\mathscr F|_{X_{\bar s}}$ at $x$ which is an integer number \cite[19.3]{linrep}. The sum
\begin{equation}\label{introsumdimtot}
\sum_{x\in D_{\bar s}}\mathrm{dt}_x(j_!\mathscr F|_{X_{\bar s}})
\end{equation}
does not depend on the choice of $\bar s$ above $s$. It defines a function $\varphi:S\rightarrow \mathbb Z$. The following property of $\varphi$ is due to Deligne and Laumon:

\begin{theorem}[{\cite[2.1.1]{lau}}]\label{themdelignelaumon}
We take the notation and assumptions of \ref{delignelaumon}. Then,
\begin{itemize}\itemsep=0.2cm
\item[(1)]
The function $\varphi:S\rightarrow \mathbb Z$ is constructible and lower semi-continuous on $S$.
\item[(2)]
If $\varphi:S\rightarrow \mathbb Z$ is locally constant, then $f:X\rightarrow S$ is universally locally acyclic with respect to $j_!\mathscr F$.
\end{itemize}
\end{theorem}

\subsection{}
Let $K$ be a complete discrete valuation field, $\mathscr O_K$ its integer ring and $F$ the residue field of $\mathscr O_K$. We assume that the characteristic of $F$ is $p>0$. Let $\overline K$ be a separable closure of $K$ and we denote by $G_K$ the Galois group of $\overline K/K$. Abbes and Saito defined two decreasing filtrations $G^r_K$ $(r\in \mathbb Q_{\geq 1})$ and $G^s_{K,\log}$ $(s\in\mathbb Q_{\geq 0})$ of $G_K$, called the ramification filtration and the logarithmic ramification filtration, respectively (\cite{as1,as2}). For any $r\in\mathbb Q_{\geq 0}$, we have $G^{r+1}_{K,\log}\subseteq G^{r+1}_K\subseteq G^{r}_{K,\log}$. If the residue field $F$ is perfect, then $G_K^{s+1}=G^s_{K,\log}$, for any $s\in\mathbb Q_{\geq 0}$, and the logarithmic ramification filtration coincides with the classical upper numbering filtration \cite[3.7]{as1}.

\subsection{}\label{introlc+1>c>lc}
Let $\ell$ be a prime number different from $p$ and $\Lambda$ a finite field of characteristic $\ell$. Let $M$ be a finitely generated $\Lambda$-module on which the wild inertia subgroup of $G_K$ acts through a finite quotient. We have two decompositions of $M$ relative to the two filtrations above
$$M=\bigoplus_{r\in \mathbb Q_{\geq 1}}M^{(r)}\ \ \ \text{and}\ \ \ M=\bigoplus_{s\in \mathbb Q_{\geq 0}}M_{\log}^{(s)},$$ 
which are called {\it the slope decomposition} and {\it the logarithmic slope decomposition} of $M$. We have the following two invariants
\begin{equation}\label{introdtmod}
\mathrm{dt}_KM=\sum_{r\geq 1}r\cdot\dim_{\Lambda}M^{(r)}\ \ \ \text{and}\ \ \ \mathrm{sw}_KM=\sum_{s\geq 0}r\cdot\dim_{\Lambda}M_{\log}^{(s)}, 
\end{equation}
called the {\it total dimension} and the {\it Swan conductor} of $M$, respectively. They generalize the classical Swan conductor and the classical total dimension. In this article, we focus on the following two invariants
\begin{align}
\rc_K(M)=\max\{r\in\mathbb Q\,|\,M^{(r)}\neq 0\}\ \ \ \text{and}\ \ \ \rlc_K(M)=\max\{s\in\mathbb Q\,|\,M^{(s)}_{\log}\neq 0\},
\end{align}
which are called the {\it conductor} and the {\it logarithmic conductor} of $M$. By Abbes and Saito's ramification theory, we have 
\begin{align*}
\sw_K(M)&\leq\dt_K(M)\leq \sw_K(M)+\dim_{\Lambda}M,\\
\rlc_K(M)&\leq\rc_K(M)\leq \rlc_K(M)+1.
\end{align*}

If the residue field $F$ is perfect, we have $$\dt_K(M)=\sw_K(M)+\dim_{\Lambda}M,\ \ \ \textrm{and}\ \ \ \rc_K(M)=\rlc_K(M)+1,$$ and the $\sw_K(M)$ (resp. $\rlc_K(M)$) is the classical Swan conductor (resp. the largest upper numbering slope) of $M$.

\subsection{}\label{YEW}
Let $\kappa$ be a field of characteristic $p>0$. Let $Y$ be a smooth $\kappa$-scheme, $E$ a reduced effective Cartier divisor on $Y$, $\{E_i\}_{i\in I}$ the set of irreducible components of $E$, $W$ the complement of $E$ in $Y$ and $h:W\rightarrow Y$ the canonical injection. We assume that each $E_i$ is generically smooth over $\mathrm{Spec}(\kappa)$. We choose an algebraic closure $\bar\kappa$ of $\kappa$. We denote by $\xi_{i}$ the generic point of an irreducible component of $E_{i,\bar\kappa}=E_i\times_{\kappa}\bar\kappa$, by $\bar\xi_{i}$ a geometric point above $\xi_{i}$, by $\eta_{i}$ the generic point of the strict localization $Y_{\bar\kappa,(\bar\xi_{i})}$, by $K_{i}$ the function field of  $Y_{\bar\kappa,(\bar\xi_{i})}$ and by $\overline K_i$ a separable closure of $K_i$.
Let $\Lambda$ be a finite field of characteristic $\ell$ $(\ell\neq p)$ and $\mathscr G$ a locally constant and constructible sheaf of $\Lambda$-modules on $W$.
The restriction $\mathscr G|_{\eta_i}$ corresponds to a finitely generated $\Lambda$-module with a continuous $\mathrm{Gal}(\overline K_i/K_i)$-action.
Since the $\mathrm{Gal}(\bar\kappa/\kappa)$-action on the set of irreducible components of $E_{i,\bar\kappa}$ is transitive, four ramification invariants  $\mathrm{dt}_{K_{i}}(\mathscr G|_{\eta_{i}})$, $\mathrm{sw}_{K_{i}}(\mathscr G|_{\eta_{i}})$, $\mathrm{c}_{K_{i}}(\mathscr G|_{\eta_{i}})$ and $\mathrm{lc}_{K_{i}}(\mathscr G|_{\eta_{i}})$ do not depend on the choice of $\bar\kappa$ nor on the choice of the irreducible component of $E_{i,\bar\kappa}$. We usually denote by $\mathrm{dt}_{E_{i}}(h_!\mathscr G)$ (resp. $\mathrm{sw}_{E_{i}}(h_!\mathscr G)$, $\mathrm{c}_{E_{i}}(h_!\mathscr G)$ and $\mathrm{lc}_{E_{i}}(h_!\mathscr G)$) instead of $\mathrm{dt}_{K_{i}}(\mathscr G|_{\eta_{i}})$ (resp. $\mathrm{sw}_{K_{i}}(\mathscr G|_{\eta_{i}})$, $\mathrm{c}_{K_{i}}(\mathscr G|_{\eta_{i}})$ and $\mathrm{lc}_{K_{i}}(\mathscr G|_{\eta_{i}})$)

We define the {\it total dimension divisor} of $h_!\mathscr G$ on  $Y$ by the Cartier divisor (\cite[Definition 3.5]{wr})
\begin{equation}\label{introdefdimtot}
\mathrm{DT}_{Y}(h_!\mathscr G)=\sum_{i\in I}\mathrm{dt}_{K_i}(\mathscr G|_{\eta_i})\cdot E_i.
\end{equation}
We define the {\it Swan divisor} of $h_!\mathscr G$ on $Y$ by the Cartier divisor
\begin{equation}
\mathrm{SW}_{Y}(h_!\mathscr G)=\sum_{i\in I}\mathrm{sw}_{K_i}(\mathscr G|_{\eta_i})\cdot E_i.
\end{equation}
We define the {\it conductor divisor} of $h_!\mathscr G$ on $Y$ by the Cartier divisor with rational coefficients
\begin{equation}
\mathrm{C}_{Y}(h_!\mathscr G)=\sum_{i\in I}\mathrm{c}_{K_i}(\mathscr G|_{\eta_i})\cdot E_i.
\end{equation}
We define the {\it logarithmic conductor divisor} of $h_!\mathscr G$ on  $Y$ by the Cartier divisor with rational coefficients
\begin{equation}
\mathrm{LC}_{Y}(h_!\mathscr G)=\sum_{i\in I}\mathrm{lc}_{K_i}(\mathscr G|_{\eta_i})\cdot E_i.
\end{equation}

\subsection{}
Assume that $f:X\rightarrow S$ has relative dimension $\geq1$, and $D$ is a Cartier divisor on $X$ relative to $S$. For each algebraic geometric point $\bar s$ of $S$, the ramification of $(j_!\mathscr F)|_{X_{\bar s}}$ at the generic points of  $D_{\bar s}$ defines four divisors supported on $D_{\bar s}$ (cf. subsection \eqref{YEW}). The lower semi-continuity for both total dimensions divisors and Swan divisors has been proved in a geometric situation (\cite[Theorem 4.3]{HY17}, \cite[Theorem 1.11]{H19}). They generalized Deligne and Laumon's result (Theorem \ref{themdelignelaumon} (1)) to higher relative dimensions in a geometric situation. 

The $D$-modules theory shares similarities with the $\ell$-adic cohomology theory. Motivated by Deligne and Laumon's result, Malgrange conjectured that irregularities of meromorphic connections on complex relative curves is lower semi-continuous. Andr\'e proved this conjecture (\cite[Corollaire 7.1.2]{An}), as well as the  lower semi-continuity for Poincar\'e-Katz ranks of meromorphic connections in the same geometric setting (\cite[Corollaire 6.1.3]{An}). In \cite[Theorem 7.2]{H23}, a semi-continuity result was proved for conductors of \'etale sheaves on relative curves of positive characteristic, which can be considered as an $\ell$-adic analogue of Andr\'e's semi-continuity result for Poincar\'e-Katz ranks. This article is devoted to proving the semi-continuity of conductor divisors and logarithmic conductor divisors for \'etale sheaves in a geometric situation, that generalizes {\it loc. cit.} to higher relative dimensions.

\subsection{}\label{nothigherdim}
Let $k$ be a perfect field of characteristic $p>0$, $S$ an irreducible $k$-scheme of finite type, $f:X\rightarrow S$ a separated and smooth $k$-morphism of finite type, $D$ an effective Cartier divisor of $X$ relative to $S$ (\cite[IV, 21.15.2]{EGA4}), $U$ the complement of $D$ in $X$ and $j:U\rightarrow X$ the canonical injection. We assume that $D$ is the sum of effective Cartier divisors $\{D_i\}_{i\in I}$ of $X$ relative to $S$, where each $D_i$ is irreducible and each restriction morphism $f|_{D_i}:D_i\to S$ is smooth. For each $s\in S$, we denote by $\rho_{s}:X_{s}\rightarrow X$ the base change of $s\rightarrow S$ by $f:X\rightarrow S$. We denote by $\eta$ the generic point of $S$. 

Let $\mathscr F$ be a locally constant and constructible sheaf of $\Lambda$-modules on $U$. We define the {\it generic conductor divisor} of $j_!\mathscr F$ on $X$ and denote by $\rGC_f(j_!\mathscr F)$ the unique Cartier divisor with rational coefficients of $X$ supported on $D$ such that $\rho_{\eta}^*(\rGC_f(j_!\mathscr F))=\rC_{X_{\eta}}(j_!\mathscr F|_{X_{\eta}})$. We define the {\it generic logarithmic conductor divisor} of $j_!\mathscr F$ on $X$ and denote by $\rGLC_f(j_!\mathscr F)$ the unique Cartier divisor with rational coefficients of $X$ supported on $D$ such that $\rho_{\eta}^*(\rGLC_f(j_!\mathscr F))=\rLC_{X_{\eta}}(j_!\mathscr F|_{X_{\eta}})$. The main result of the article is the following:

\begin{theorem}\label{intromaintheorem}
We take the notation and assumptions of \ref{nothigherdim}. Then,  
\begin{itemize}\itemsep=0.2cm
\item[(i)]
there exists an open dense subset $V$ of $S$ such that, for any point $s\in V$, we have
\begin{equation*}
\rho_{ s}^*(\mathrm{GC}_{f}(j_!\mathscr F))=\mathrm{C}_{X_{s}}((j_!\mathscr F)|_{X_{s}})
\end{equation*}
and, for any point $t\in S-V$, we have 
\begin{equation*}
\rho_{t}^*(\mathrm{GC}_{f}(j_!\mathscr F))\geq\mathrm{C}_{X_{t}}((j_!\mathscr F)|_{X_{t}}).
\end{equation*}
\item[(ii)]
there exists an open dense subset $W$ of $S$ such that for any point $s\in W$, we have
\begin{equation*}
\rho_{s}^*(\mathrm{GLC}_{f}(j_!\mathscr F)+D)=\mathrm{LC}_{X_{s}}(j_!\mathscr F|_{X_{s}})+(D_{s})_{\mathrm{red}}
\end{equation*}
and, for any point $t\in S-W$, we have 
\begin{equation*}
\rho_{t}^*(\mathrm{GLC}_{f}(j_!\mathscr F)+D)\geq\mathrm{LC}_{X_{s}}((j_!\mathscr F)|_{X_{t}})+(D_{t})_{\mathrm{red}}.
\end{equation*}
\end{itemize}
\end{theorem}

\begin{remark}
The proof of Theorem \ref{intromaintheorem} (i) follows a strategy similar to that of \cite[Theorem 4.3]{HY17}. In the study of the semi-continuity for Swan divisors on higher relative dimension by the first named author \cite[Theorem 1.11]{H19}, two main ingredients are 
\begin{itemize}
\item[a)]
Deligne and Laumon's semi-continuity result for Swan conductors of \'etale sheaves on relative curves (\cite{lau});
\item[b)]
an asymptotic formula for the Swan divisor under the cutting-by-curve \cite[Theorem 6.5]{H19}. 
\end{itemize}
Besides, a Hasse-Arf theorem (\cite[Theorem 3.4.3]{xiao}) for the Swan conductor defined by Abbes and Saito's logarithmic ramification filtration is indispensable in the proof. In the proof of Theorem \ref{intromaintheorem} (ii), two main ingredients parallel to a) and b) have been developed in \cite{H23}. By contrast, there is no Hasse-Arf theorem for the logarithmic conductor divisor. The key point is to give an upper bound of denominators for logarithmic conductors relating to the sheaf (Lemma \ref{Hasse-Arf-conductor}) and the bound is indispensable in the proof. Therefore, the proof for Theorem \ref{intromaintheorem} (ii) is different from that for \cite[Theorem 1.11]{H19}. Both authors thank the anonymous referee for suggesting an axiomatization toward a semi-continuity theorem for general constructible functions on varieties. We are sorry and we have no idea to this direction.

\end{remark}

\subsection{}
Theorem \ref{intromaintheorem} and Corollary \ref{coro_for_Betti_bound} are the crucial ingredients in \cite{HT24} to get estimates for the Betti numbers of  perverse sheaves with bounded rank  and wild ramification \textit{in families} rather than on a fixed variety.

\subsection*{Acknowledgement}
Both authors would like to thank A. Abbes and Y. Cao for stimulating discussions and valuable comments. We are grateful to the anonymous referee for a careful reading and for many useful comments and suggestions. Some part of this work was done during the first named author's stay at IHES in July 2024. Both authors would like to thank the institute for the hospitality. H. H. is supported by the National Natural Science Foundation of China (Grant No. 12471012) and the Natural Science Foundation of Jiangsu Province (Grant No. BK20231539).

\section{Notation}
\subsection{}
In this article, let $k$ be a field of characteristic $p>0$. We fix a prime number $\ell$ which is different from $p$, and a finite field $\Lambda$ of characteristic $\ell$. All $k$-schemes are assumed to be separated and of finite type over $\Spec(k)$ and all morphisms between $k$-schemes are assumed to be $k$-morphisms. All sheaves of $\Lambda$-modules on $k$-schemes are assumed to be \'etale sheaves.

\subsection{}\label{bigger}
Let $X$ be a Noetherian scheme and $\mathrm{CDiv}(X)$ the $\mathbb Z$-modules of Cartier divisors on $X$. {\it A Cartier divisor with rational coefficients on $X$} denotes an element in $\mathrm{CDiv}(X)\otimes_{\mathbb Z}\mathbb Q$. We say that $E_1$ is bigger than $E_2$ and we write $E_1\geq E_2$ for two elements $E_1,E_2\in \mathrm{CDiv}(X)\otimes_{\mathbb Z}\mathbb Q$, if there exists an positive integer $r$ such that $r(E_1-E_2)$ is an effective Cartier divisor on $X$.

\subsection{}
Let $f:X\rightarrow S$ be a morphism of schemes, $s$ a point of $S$, and $\bar s\rightarrow S$ a geometric point above $s$. We denote by $X_s$ (resp. $X_{\bar s}$) the fiber $X\times_Ss$ (resp. $X\times_S\bar s$). Assume that $f:X\rightarrow S$ is flat and of finite presentation. Let $D$ be a Cartier divisor on $X$ relative to $S$ (\cite[IV, 21.15.2]{EGA4}). Let $\pi:S'\rightarrow S$ be a morphism of $k$-schemes, $X'=X\times_SS'$ and $\pi':X'\rightarrow X$ the base change of $\pi:S'\rightarrow S$. We denote by $\pi'^*D$ the pull-back of $D$, which is a Cartier divisor on $X'$ relative to $S'$  \cite[IV, 21.15.9]{EGA4}. When $S'$ is $s$ or $\bar s$, we simply denote by $D_s$ (resp. $D_{\bar s}$) the Cartier divisor $D\times_Ss$ on $X_s$ (resp. $D\times_S\bar s$ on $X_{\bar s}$). An effective Cartier divisor $E$ on $X$ relative to $S$ is identical to a closed immersion $i:E\rightarrow X$ transversally regular relative to $S$ and of codimension $1$ (\cite[IV,19.2.2 and 21.15.3.3]{EGA4}). The fiber $E_s$ (resp. $E_{\bar s}$) is an effective Cartier divisor on $X_s$ (resp. $X_{\bar s}$). 

Let $\{D_i\}_{i\in I}$ be a set of effective Cartier divisors of $X$ relative to $S$.  A linear combination $Q=\sum_{i\in I}r_iD_i$ with $r_i\in\mathbb Q$ is called a {\it Cartier divisor of $X$ relative to $S$ with rational coefficients} supported on $D=\sum_{i\in I}D_i$. We denote by $\pi'^*Q$ the linear combination $\sum_{i\in I}r_i(\pi'^*D_i)$, which is a Cartier divisor of $X'$ relative to $S'$ with rational coefficients. When $S'$ is $s$ or $\bar s$, the fiber $Q_s=\sum_{i\in I}r_iD_{i,s}$ (resp. $Q_{\overline s}=\sum_{i\in I}r_iD_{i,\overline s}$) is a Cartier divisor on $X_s$ (resp. $X_{\overline s}$) with rational coefficients.

\subsection{}
 Let $x$ be a closed point of a $k$-scheme $X$. For any irreducible closed subscheme $Z$ of $X$ containing $x$, we denote by $m_x(Z)$ the multiplicity of $Z$ at $x$ (\cite[4.3]{ful}).

\subsection{}
Let $X$ be a smooth $k$-scheme. We denote by $\bT X$ the tangent bundle of $X$ and by $\bT^*X$ the cotangent bundle of $X$. By a {\it closed conical subset} of $\bT^*X$, we mean a reduced closed subscheme of $\bT^*X$ invariant under the canonical $\mathbb G_m$-action on $\bT^*X$. For any point $x$ of $X$, we put $\bT_xX=\bT X\times_Xx$ and $\bT^*_xX=\bT^*X\times_Xx$.

%Let $Z$ be a  smooth $k$-scheme, and $i:Z\rightarrow X$ a closed immersion. We denote by $\bT^*_ZX$ the conormal bundle of $Z$ in $X$.

%\subsection{}\label{interproper}
%Let $X$ be a scheme over a field of finite type, $f:Y\rightarrow X$ a morphism of finite type and $g:Z\rightarrow X$ a regular immersion of codimension $c$. We say $Z$ {\it intersects}  $Y$ {\it properly} if, for each irreducible component $Y'$ of $Y$, the fiber product $Z'=Y'\times_XZ$ is equidimensional and the canonical injection $Z'\rightarrow Y'$ is of codimension $c$.

\section{Preliminaries in geometry}

\begin{proposition}[cf. {\cite[I, Chapitre 0, 15.1.16]{EGA4}}]\label{flatlemma}
Let $A\rightarrow B$ be a local homomorphism of noetherian local rings, $\kappa$ the residue field of $A$, $b$ an element of the maximal ideal of $B$ and $\bar b\in B\times_A\kappa$ the residue class of $b$. Then, the following conditions are equivalent:
\begin{itemize}\itemsep=0.2cm
\item[(1)]
The quotient $B/bB$ is flat over $A$ and $b$ is a non-zero divisor of $B$;
\item[(2)]
$B$ is flat over $A$ and $\bar b$ is a non-zero divisor of $B\otimes_A\kappa$.
\end{itemize}
\end{proposition}
It is deduced by the equivalence of (b) and (c) of \cite[I, Chapitre 0, 15.1.16]{EGA4}.

\begin{proposition}[{\cite[IV, 18.12.1]{EGA4}}]\label{zarmain}
Let $S$ and $D$ be schemes, $f:D\rightarrow S$ a separated morphism locally of finite type, $x$ a point of $D$ and $s=f(x)$. We assume that $x$ is an isolated point of $f^{-1}(s)$. Then, there exists an \'etale morphism $S'\rightarrow S$, a point $x'$ of $D'=D\times_SS'$ above $x$ and a Zariski open and closed neighborhood $V'$ of $x'$ in $D'$ such that  $V'$ is finite over $S'$ and $f'^{-1}(f'(x'))\cap V'=\{x'\}$, where $f':D'\rightarrow S'$ denotes the base change of $f:D\rightarrow S$.
\end{proposition}

\begin{proposition}[{\cite[Proposition 5.3]{HY17}}]\label{transversalcurve}
Let $\kappa$ be a field with infinitely many elements, $X$ a connected smooth $\kappa$-scheme of dimension $n\geq 2$, $D$ an effective Cartier divisor on $X$ which is smooth at a $\kappa$-rational point $x\in D$ and $S\subseteq \mathbb T^*_xX$ a closed conical subset of dimension $1$. Then, we can find a smooth $\kappa$-curve $C$ and an immersion $h:C\rightarrow X$ such that $C$ intersects $D$ transversally at $x$ and that $\ker(dh)\cap S=\{0\}$, where $dh:\mathbb T^*_xX\rightarrow\mathbb T^*_xC$ is the canonical map.
\end{proposition}

\begin{proposition}\label{goodsection}
Let $S$ be an irreducible separated scheme over an algebraically closed field $\kappa$ of finite type and $s\in S$ a closed point such that $S$ is regular at $s$. Let $Z\subseteq\mathbb A^n_S$ be a closed subscheme with $Z\neq \mathbb A^n_S$ containing $s\times O$, where $O$ denotes the origin of $\mathbb A^n_\kappa$ $(n\geq 1)$. Then, there exist an open dense subscheme $V$ of $S$ and an immersion $\sigma:V\to\mathbb A^n_S$ such that $s\times O\in\sigma(V)$, $V\not\subset Z$ and that the composition of $\sigma:V\to\mathbb A^n_S$ and $\pi:\mathbb A^n_S\to S$ is the canonical injection.
\end{proposition}

\begin{proof}
We denote by $\eta$ the generic point of $S$.
Let $y$ be a non-zero element of the maximal ideal $\mathfrak m_{S,s}$ of $\mathscr O_{S,s}$. For $\underline\lambda=(\lambda_1,\lambda_2,\dots,\lambda_n)\in \kappa^n$, we define a homomorphism of $\mathscr O_{S,s}$-algebras 
\begin{align}
g_{\underline\lambda}:\mathscr O_{S,s}[x_1,\dots,x_n]&\to \mathscr O_{S,s},\ \ \ x_i\mapsto \lambda_iy.
\end{align}
We can find an open dense subset $V$ of $S$ such that, for any $\underline\lambda\in\kappa^n$, the homomorphism $g_{\underline\lambda}$ gives rise to a $\kappa$-morphism $h_{\underline\lambda}:V\to\bA^n_S$. The composition of each $h_{\underline\lambda}:V\to\bA^n_S$ and $\pi:\mathbb A^n_S\to S$ is the canonical injection $V\to S$, and $h_{\underline\lambda}(s)=s\times O$. The image $h_{\underline\lambda}(\eta)\subset \bA^n_{\eta}$ of the generic point $\eta\in V$ is the closed point $x_{\underline\lambda}=(\lambda_1y,\dots,\lambda_ny)\in\bA^n_{\eta}(k(\eta))$. 

Note that $\kappa$ is algebraically closed. Thus the set of closed points 
$\{x_{\underline\lambda}\;|\;\underline\lambda\in\kappa^n\}\subset \bA^n_{\eta}(k(\eta))$
is dense in $\mathbb A^n_\eta$. Since $Z\neq\bA^n_S$, the intersection $Z\bigcap \mathbb A^n_{\eta}$ is a closed subscheme of $\mathbb A^n_{\eta}$ which does not contain the generic point of $\mathbb A^n_{\eta}$. Hence, there exists $\underline\lambda\in \kappa^n$ such that $x_{\underline\lambda}\not\in Z$. The associated immersion $h_{\underline\lambda}:V\to\mathbb A^n_S$ satisfies the conditions of the proposition.
\end{proof}

\begin{proposition}[cf. {\cite[III, 9.5.3]{EGA4}}]\label{genetosp2}
Let $S$ be an irreducible noetherian scheme, $g:D\rightarrow S$ a morphism of finite type, $\{D_i\}_{i\in I}$  the set of irreducible components of $D$. We assume that, for each $i\in I$, the restriction $f|_{D_i}:D_i\rightarrow S$ is surjective. Then there exists an open dense subset $V\subseteq S$ such that, for every point $v\in V$ and for any indices $i, j\in I$ $(i\neq j)$, the fibers $D_{i,v}$ and $D_{j,v}$ do not have common irreducible components.
\end{proposition}

\begin{proposition}[cf. {\cite[III, 9.7.7]{EGA4}}]\label{genetosp1}
Let $S$ and $D$ be integral $k$-schemes and $f:D\rightarrow S$ a smooth $k$-morphism. Then there exist an irreducible $k$-scheme $W$ and an \'etale map $h:W\rightarrow S$ such that,
\begin{itemize}\itemsep=0.2cm
\item[(1)]
Each connected component of $D\times_SW$ is irreducible;
\item[(2)]
Every connected component of  $D\times_SW$ has geometrically irreducible fibers over $W$.
\end{itemize}
\end{proposition}

%\begin{corollary}\label{compcycle}
%Let $S$ be an irreducible $k$-scheme, $f:X\rightarrow S$ a smooth $k$-morphism, $\{D_i\}_{i\in I}$ a set of effective Cartier divisors on $X$ relative to $S$, $D$ the sum of all $D_i$ $(i\in I)$. For every $i\in I$, we assume that $D_i$ is irreducible and that $f|_{D_i}:D_i\rightarrow S$ is smooth.
% Let $R=\sum_{i\in I} r_i\cdot D_i$ be a Cartier divisor on $X$. If there exists an open dense subset $V\subseteq S$ such that, for each closed point $t\in V$, the fiber $R_t$ is an effective Cartier divisor on $X_t$ (\ref{fiberschdiv}), then $R$ is also effective.
%\end{corollary}
%\begin{proof}
%We may assume that $S$ is integral. By proposition \ref{genetosp2}, this problem can be  reduced to the case where $D$ is irreducible. Let $h:W\rightarrow S$ be an \'etale morphism and we denote by $h_W:X_W=X\times_SW\rightarrow X$ its base change. Then $R$ is effective if and only if $h^*_WR$ is effective. Hence, replacing $S$ by an \'etale neighborhood, we may assume that, for any point $s\in S$, the fiber $D_s$ is geometrically irreducible (proposition \ref{genetosp1}). In this case, for any closed point $s\in S$ such that $D_s\neq \emptyset$, the coefficient of $D_s$ in $R_s$ equals the coefficient of $D$ in $R$.
%\end{proof}

\section{Complements in ramification theory}
\begin{lemma}\label{Hasse-Arf-conductor}
Let $K$ be a complete discrete valuation field of characteristic $p>0$, $\mathscr O_K$ its integer ring and $F$ the residue field of $\mathscr O_K$. Let $\overline K$ be a separable closure of $K$ and we denote by $G_K$ the Galois group of $\overline K/K$. Let $M$ be a non-zero finitely generated $\Lambda$-modules with a continuous $G_K$-action. Then, there exist integers $0<r_1,r_2\leq\mathrm{rk}_{\Lambda}M$ such that 
\begin{equation*}
\rc_K(M)\in \frac{1}{r_1}\mathbb Z\ \ \ \textrm{and}\ \ \ \rlc_K(M)\in \frac{1}{r_2}\mathbb Z.
\end{equation*}
\end{lemma}
\begin{proof}
Let $M=\bigoplus_{r\geq 1}M^{(r)}$ (resp. $M=\bigoplus_{s\geq 0}M^{(s)}_{\log}$) be the slope decomposition (resp. logarithmic slope decomposition) of $M$ . Each $M^{(r)}$ is a finite generated $\Lambda$-module with continuous $G_K$-actions. By the Hasse-Arf theorem for Abbes-Saito's ramification filtrations (\cite[Theorem 3.4.3]{xiao}, \cite[Proposition 3.10]{wr}), we obtain that
\begin{align*}
\dt_K(M^{(\rc_K(M))})&=\rc_K(M)\cdot\mathrm{rk}_\Lambda \left(M^{(\rc_K(M))}\right),\\sw_K(M^{(\rlc_K(M))})&=\rlc_K(M)\cdot\mathrm{rk}_\Lambda \left(M^{(\rlc_K(M))}_{\log}\right)
\end{align*}
are integers. Put $r_1=\mathrm{rk}_\Lambda \left(M^{(\rc_K(M))}\right)$ and $r_2=\mathrm{rk}_\Lambda \left(M^{(\rlc_K(M))}_{\log}\right)$. We have 
\begin{equation*}
\rc_K(M)\in \frac{1}{r_1}\mathbb Z\ \ \ \textrm{and}\ \ \ \rlc_K(M)\in \frac{1}{r_2}\mathbb Z.
\end{equation*}
\end{proof}

\begin{proposition}[{\cite[Proposition 3.15 (3)]{as1}}]\label{logext}
Let $K$ be a complete discrete valuation field, $\mathscr O_K$ its integer ring and $F$ the residue field of $\mathscr O_K$. We assume that the characteristic of $F$ is $p>0$.  Let $K'$ be a finite separable extension of $K$ contained in $\overline K$ of ramification index $e$. We denote by $G_{K'}$ the Galois group of $\overline K$ over $K'$ and $G_{K',\log}^r$ ($r\in \mathbb Q_{\geq 0}$) the logarithmic ramification filtration of $G_{K'}$. Then, for any $r\in \mathbb Q_{> 0}$, we have $G^{er}_{K',\log}\subseteq G^r_{K,\log}$. If $K'$ is tamely ramified over $K$, the inclusion is an equality.
\end{proposition}

\begin{lemma}\label{bcofclc}
Let $X$ be a smooth $k$-scheme, $D$ a reduced Cartier divisor of $X$,  $U$ the complement of $D$ in $X$ and  $j\colon U\to X$ the canonical injection. Let $\sF$ be a locally constant and constructible \'etale sheaf of $\Lambda$-modules on $U$.  Let $k'$ be an algebraic extension of $k$, and $g:X_{k'}=X\times_{k}k'\rightarrow X$ the canonical projection. We have
 \begin{equation}\label{basechangeclc}
g^*(\mathrm{C}_X(j_!\mathscr F))=\mathrm{C}_{X_{k'}}(g^*j_!\mathscr F)\ \ \ \text{and}\ \ \ g^*(\mathrm{LC}_X(j_!\mathscr F))=\mathrm{LC}_{X_{k'}}(g^*j_!\mathscr F).
 \end{equation}
\end{lemma}

This lemma is a direct consequence from the definition of the (logarithmic) conductor divisor.

\begin{lemma}\label{bcofgcglc}
We take the notation and assumptions in subsection \ref{nothigherdim}.
\begin{itemize}\itemsep=0.2cm
\item[(i)]
Let $S'$ be an irreducible smooth $k$-scheme of finite type and $\pi:S'\rightarrow S$ a dominant and generically finite $k$-morphism. We denote by
$f':X'\rightarrow S'$ (resp. by $\pi':X'\rightarrow X$) the base change of $f:X\rightarrow S$ (resp. $\pi:S'\rightarrow S$). We have
\begin{equation}\label{finitepbgdt}
\pi'^*(\mathrm{GC}_f(j_!\mathscr F))=\mathrm{GC}_{f'}(\pi'^*j_!\mathscr F)\ \ \ \text{and}\ \ \ \pi'^*(\mathrm{GLC}_f(j_!\mathscr F))=\mathrm{GLC}_{f'}(\pi'^*j_!\mathscr F).
\end{equation}
\item[(ii)]
Let $\overline k$ be an algebraic closure of $k$, $\overline S$ an irreducible component of $S\otimes_{k}\overline k$ and $\pi:\overline S\to S$ the composition of $\overline S\to S\otimes_{k}\overline k$ and $S\otimes_{k}\overline k\to S$. We denote by
$\overline f:\overline X\rightarrow \overline S$ (resp. by $\overline\pi:\overline X\rightarrow X$) the base change of $f:X\rightarrow S$ (resp. $\pi:\overline S\rightarrow S$). We have 
\begin{equation}\label{algclosedpbgdt}
\overline\pi^*(\mathrm{GC}_f(j_!\mathscr F))=\mathrm{GC}_{\overline f}(\overline\pi^*j_!\mathscr F)\ \ \ \text{and}\ \ \ \overline\pi^*(\mathrm{GLC}_f(j_!\mathscr F))=\mathrm{GLC}_{\overline f}(\overline\pi^*j_!\mathscr F).
\end{equation}
\item[(iii)]
Let $g:X'\to X$ be an \'etale morphism such that $X'_{\eta}=X'\times_S\eta\neq \emptyset$. Then, we have 
\begin{equation}
g^*(\rGC_f(j_!\mathscr F))=\rGC_{fg}(j_!\mathscr F|_{X'})\ \ \ \textrm{and}\ \ \ g^*(\rGLC_f(j_!\mathscr F))=\rGLC_{fg}(j_!\mathscr F|_{X'}).
\end{equation}
\end{itemize}
\end{lemma}

Lemma \ref{bcofgcglc} (i) and (ii) follow from the definition of the (logarithmic) conductor divisor and Lemma \ref{bcofclc}. Lemma \ref{bcofgcglc}(iii) follows from the fact that $g_{\eta}:X'_{\eta}\to X_{\eta}$ is \'etale and the (logarithmic) conductor divisor is an \'etale local invariant.

\begin{proposition}\label{Ssmoothoveralgclosed}
We take the notation and assumptions in subsection \ref{nothigherdim}.
To prove Theorem \ref{intromaintheorem}, it is sufficient to consider the case where $k$ is algebraically closed and $S$ is a connected and smooth $k$-scheme.
\end{proposition}
\begin{proof}
Let $k'$ be an algebraic extension of $k$, let $S'$ be an irreducible separable $k'$-scheme of finite type and let $\theta:S'\to S$ be a $k$-morphism satisfying
\begin{itemize}\itemsep=0.2cm
\item[(1)]
$\theta:S'\to S$ is surjective, and, for any open dense subset $V'\subset S'$, the image $\theta(V')$ contains an open dense subset of $S$;
\item[(2)]
For each $i\in I$, all connected components of $D'_i=D_i\times_SS'$ are smooth over $S'$.
\end{itemize}
We denote by $f':X'\to S'$ and $\theta':X'\to X$ the base changes of $f:X\to S$ and $\theta:S'\to S$. Let $s'$ be a point of $S'$ and $s=\theta(s')\in S$. We denote by $\rho'_{s'}:X'_{s'}=X'\times_{S'}s'\to X'$ the canonical injection and by $\iota_{s'}:X'_{s'}\to X_s$ the canonical projection.

Assume that Theorem \ref{intromaintheorem}(i) is valid for the morphism $f':X'\to S'$ and the sheaf $\theta'^*j_!\sF$. Then, there exists an open dense subset $V'\subset S'$ such that
\begin{itemize}\itemsep=0.2cm
\item{} for any $s'\in V'$, we have $\rho'^*_{s'}(\rGC_{f'}(\theta'^*j_!\sF))=\rC_{X'_{s'}}(\rho'^*_{s'}\theta'^*j_!\sF)$;
\item{} for any $s'\in S'-V'$, we have $\rho'^*_{s'}(\rGC_{f'}(\theta'^*j_!\sF))\geq\rC_{X'_{s'}}(\rho'^*_{s'}\theta'^*j_!\sF)$.
\end{itemize}
Let $V$ be an open dense subset of $S$ contained in $\theta(V')$. Then, for any point $s\in V$ and for a point $s'\in V'$ above $s\in V$ such that $k(s')/k(s)$ is algebraic, we have 
\begin{align}\label{iotarhoGC=iotaC}
\iota_{s'}^*\rho_s^*(\rGC_f(j_!\sF))&=\rho'^*_{s'}\theta'^*(\rGC_f(j_!\sF))=\rho'^*_{s'}(\rGC_{f'}(\theta'^*j_!\sF))\\
&=\rC_{X'_{s'}}(\rho'^*_{s'}\theta'^*j_!\sF)=\rC_{X'_{s'}}(\iota^*_{s'}\rho^*_sj_!\sF)=\iota^*_{s'}(\rC_{X_{s}}(\rho^*_sj_!\sF)),\nonumber
\end{align}
where the second equality follows from Lemma \ref{bcofgcglc} and the fifth equality follows from Lemma \ref{bcofclc}. Since $X_s$ is a smooth $k(s)$-scheme and $X'_{s'}=X_s\otimes_{k(s)}k(s')$, the equality \eqref{iotarhoGC=iotaC} of two Cartier divisors with rational coefficients on $X'_{s'}$ implies that $\rho_s^*(\rGC_f(j_!\sF))=\rho_s^*(\rGC_f(j_!\sF))$. Similarly, for any $s\in S-V$ and for a point $s'\in S'$ above $s\in S-V$ such that $k(s')/k(s)$ is algebraic, we have 
\begin{align}\label{iotarhoGC>iotaC}
\iota_{s'}^*\rho_s^*(\rGC_f(j_!\sF))&=\rho'^*_{s'}\theta'^*(\rGC_f(j_!\sF))=\rho'^*_{s'}(\rGC_{f'}(\theta'^*j_!\sF))\\
&\geq\rC_{X'_{s'}}(\rho'^*_{s'}\theta'^*j_!\sF)=\rC_{X'_{s'}}(\iota^*_{s'}\rho^*_sj_!\sF)=\iota^*_{s'}(\rC_{X_{s}}(\rho^*_sj_!\sF)), \nonumber
\end{align}
which implies $\rho_s^*(\rGC_f(j_!\sF))\geq\rho_s^*(\rGC_f(j_!\sF))$. Hence, Theorem \ref{intromaintheorem}(i) is valid for the morphism $f:X\to S$ and the sheaf $j_!\sF$. 

Let $\{S_\alpha\}_{1\leq \alpha\leq m}$ be an open affine cover of $S$. Note that, to prove Theorem \ref{intromaintheorem}(i), it is sufficient to verify it for the morphisms $f_\alpha:X\times_SS_\alpha\to S_\alpha$ and the sheaves $(j_!\mathscr F)|_{X\times_SS_\alpha}$ for each $1\leq \alpha\leq m$. Hence, we firstly reduced to the case that $S$ is an irreducible and affine $k$-scheme of finite type. Using the argument in the previous paragraph, we secondly replace $S$ by $S_{\mathrm{red}}$, and thirdly reduced to the case that $S$ is an irreducible and smooth $k$-scheme by de Jong's alteration \cite[4.1]{dej}. Let $\overline k$ be an algebraic closure of $k$, $S'$ an irreducible component. Finally, we can replace $S$ by $S'$ and reduce Theorem \ref{intromaintheorem}(i) to the case where $S$ is connected and smooth over an algebraically closed field.

The argument above is also valid for the logarithmic version.
\end{proof}

\subsection{\cite[1.2]{bei}}
Let $X$ be a smooth $k$-scheme and $C$ a closed conical subset in $\bT^*X$.   Let $f\colon Y\rightarrow X$ be a morphism of smooth $k$-schemes and $\ol y\rightarrow Y$ a geometric point above a point $y$ of $Y$. We say that $f\colon Y\rightarrow X$ is $C$-{\it transversal at} $y$ if $\ker(df_{\ol y})\bigcap (C\times_X\ol y)\subseteq\{0\}\subseteq \bT^*_{f(\ol y)}X$, where $df_{\ol y}\colon \bT^*_{f(\ol y)}X\rightarrow\bT^*_{\ol y}Y$ is the canonical map. We say that $f\colon Y\rightarrow X$ is $C$-{\it transversal}  if it is $C$-transversal at every point of $Y$. If $f\colon Y\rightarrow X$ is $C$-transversal, we define $f^\circ C$ to be the scheme theoretic image of $Y\times_XC$ in $\bT^*Y$ by the canonical map $df\colon  Y\times_X\bT^*X\rightarrow \bT^*Y$. Notice that $df\colon Y\times_XC\rightarrow f^\circ C$ is finite and that $f^\circ C$ is also a closed conical subset of $\bT^*Y$.
Let $g\colon X\rightarrow Z$ be a morphism of smooth $k$-schemes, $x$ a point of $X$, and $\ol x\rightarrow X$ a geometric point above $x$. We say that $g\colon X\rightarrow Z$ is $C$-transversal at $x$ if $dg_{\ol x}^{-1} (C\times_X\ol x)\subseteq\{0\}\subset \bT^*_{g(\ol x)}Z$, where $dg_{\ol x}\colon \bT^*_{g(\ol x)}Z\rightarrow \bT^*_{\ol x}X$ is the canonical map. We say that $g\colon X\rightarrow Z$ is $C$-{\it transversal} if it is $C$-transversal at every point of $X$. Let $(g,f)\colon Z\leftarrow Y\rightarrow X$ be a pair of morphisms of smooth $k$-schemes. We say that $(g,f)$ is $C$-{\it transversal} if $f\colon Y\rightarrow X$ is $C$-transversal and $g\colon Y\rightarrow Z$ is $f^\circ C$-transversal.

\subsection{\cite[1.3]{bei}}\label{NotationSS}
Let $X$ be a smooth $k$-scheme and $\sK$ an object of $D^b_c(X,\Lambda)$. We say that $\sK$ is {\it micro-supported} on a closed conical subset $C$ of $\bT^*X$ if, for any $C$-transversal pair of morphisms $(g,f)\colon Z\leftarrow Y\rightarrow X$ of smooth $k$-schemes, the morphism $g\colon Y\rightarrow Z$ is locally acyclic with respect to $f^*\sK$. If there exists a smallest closed conical subset of $\bT^*X$ on which $\sK$ is micro-supported, we call it the {\it singular support} of $\sK$ and denote it by $SS(\sK)$. We often endow $SS(\sK)$ with a reduced induced closed subscheme structure.

\begin{theorem}[{\cite[Theorem 1.3]{bei}}]\label{BeiSS}
Let $X$ be a smooth $k$-scheme and $\sK$ an object of $D^b_c(X,\Lambda)$. The singular support $SS(\sK)$ exists. Moreover, each irreducible component of $SS(\sK)$ has dimension $\dim_kX$ if $X$ is equidimensional.
\end{theorem}

\subsection{}
In the following of this section, we assume that $k$ is perfect. 
%Let $X$ be a smooth $k$-scheme, let $D$ be a reduced Cartier divisor of $X$, let $U$ be the complement of $D$ in $X$ and let $j\colon U\to X$ be the canonical injection. We denote by $\{D_i\}_{1\leq i\leq m}$ the set of irreducible components of $D$. Let $\sF$ be a locally constant and constructible \'etale sheaf of $\Lambda$-modules on $U$.

\begin{theorem}[{cf. \cite{wr, H23}}]\label{Ccutbycurve}
Let $X$ be a smooth $k$-scheme, $D$ a reduced Cartier divisor of $X$,  $U$ the complement of $D$ in $X$ and  $j\colon U\to X$ the canonical injection. Let $\sF$ be a locally constant and constructible \'etale sheaf of $\Lambda$-modules on $U$. Let $f\colon Y\to X$ be a morphism of smooth $k$-schemes. We assume that $f^*D=D\times_XY$ is a Cartier divisor of $Y$. 
\begin{itemize}\itemsep=0.2cm
\item[(1)] Then, we have {\rm(\cite[Theorem 1.5]{H23})}
\begin{equation*}
f^*(\rC_X(j_!\sF))\geq \rC_Y(f^*j_!\sF).
\end{equation*}
\item[(2)]
Assume that $Y$ is a smooth $k$-curve, that $f\colon Y\to X$ is an immersion such that $D$ is smooth at the closed point $x=Y\bigcap D$ and that the ramification of $\sF$ at $x$ is non-degenerate. If $f\colon Y\to X$ is $SS(j_!\sF)$-transversal at $x$, then \rm{(\cite[Proposition 3.8, Corollary 3.9]{wr})}
$$f^*(\rC_X(j_!\sF))= \rC_Y(f^*j_!\sF).$$
\end{itemize}
\end{theorem}

\begin{theorem}[{\cite[Theorem 1.6]{H23}}]\label{LCcutbycurve}
Let $X$ be a smooth $k$-scheme, $D$ a divisor with simple normal crossings of $X$,  $U$ the complement of $D$ in $X$ and  $j\colon U\to X$ the canonical injection. Let $\sF$ be a locally constant and constructible \'etale sheaf of $\Lambda$-modules on $U$. Let $f\colon Y\to X$ be a morphism of smooth $k$-schemes. We assume that $f^*D=D\times_XY$ is a Cartier divisor of $Y$. 

\begin{itemize}\itemsep=0.2cm
\item[(1)]
Then, we have
\begin{equation}\label{fLC>LCf}
f^*(\rLC_X(j_!\sF))\geq \rLC_Y(f^*j_!\sF).
\end{equation}
\item[(2)]
We further assume that $D$ is irreducible. Let $\mathcal I(X, D)$ be the set of triples $(S,h\colon S\to X,x)$ where $h\colon S\to X$ is an immersion from a smooth $k$-curve $S$ to $X$ such that $x=S\bigcap D$ is a closed point of $X$. Then, we have
\begin{equation}\label{equality lc sup}
\rlc_D(j_!\sF)=\sup_{\mathcal I(X, D)}\frac{\rlc_x(h^*j_!\sF)}{m_x(h^*D)}.
\end{equation}
\end{itemize}
\end{theorem}

%\subsection{}\label{notsctd}
%Let $S$ be an irreducible $k$-scheme, $f:X\rightarrow S$ a smooth $k$-morphism, $\{D_i\}_{i\in I}$ a set of effective Cartier divisors on $X$ relative to $S$, $D$ the sum of all $D_i$ $(i\in I)$, $U$ the complement of $D$ in $X$ and $j:U\rightarrow X$ the canonical injection. For every $i\in I$, we assume that $D_i$ is irreducible and that $f|_{D_i}:D_i\rightarrow S$ is smooth. For any point $t\in S$, we denote  by $\rho_t:X_t\rightarrow X$ the canonical injection. Let $\mathscr F$ be a locally constant and constructible sheaf of $\Lambda$-modules on $U$.

%\begin{theorem}
%Under the notation and assumptions of \ref{notsctd}, there exists an open dense subset $V\subseteq S$ such that:
%\begin{itemize}
%\item[(1)]
%for any point $s\in V$, we have $\rho_s^*(\mathrm{GDT}_f(j_!\mathscr F))=\mathrm{DT}_{X_{s}}(j_!\mathscr F|_{X_{s}})$;
%\item[(2)]
%for any point $t\in S-V$, we have $\rho_t^*(\mathrm{GDT}_f(j_!\mathscr F))\geq\mathrm{DT}_{X_{t}}(j_!\mathscr F|_{X_{t}})$.
%\end{itemize}
%\end{theorem}

\section{Proof of the main theorem}\label{sccd}
\subsection{} In this section, we aim to prove Theorem \ref{intromaintheorem}. We adapt a strategy similar to proofs of \cite[Theorem 4.3]{HY17} and  \cite[Theorem 1.11]{H19}.
We take the notation and assumptions of subsection \ref{nothigherdim}. We further assume that $k$ is algebraically closed and $S$ is a connected and smooth $k$-scheme from Proposition \ref{cdprop1} to subsection \ref{finalproofCandLCtheorem}, by Proposition \ref{Ssmoothoveralgclosed}. 

\begin{proposition}\label{cdprop1}
For each closed point $s\in S$, we have
\begin{equation}\label{cdprop1formula}
\rho_s^*(\mathrm{GC}_f(j_!\mathscr F))\geq\mathrm{C}_{X_{s}}(j_!\mathscr F|_{X_{s}}).
\end{equation}
\end{proposition}
\begin{proof}
We fix a closed point $s\in S$. This is a local statement for the Zariski topology of $X$. After shrinking $X$, we may assume that $X$ is affine and irreducible, that $(D_{s})_{\mathrm{red}}$ has a unique irreducible component and that, for each $i\in I$, $D_s\subseteq D_i$. Notice that, for each $i\in I$, we have  $(D_{i})_s\xrightarrow{\sim} (D_s)_{\mathrm{red}}$. We put  $n=\dim_k X-\dim_k S$.

When $n=1$, i.e., $f:X\to S$ is a smooth relative curve, the inequality \eqref{cdprop1formula} is due to \cite[Theorem 7.2]{H23}.
We consider the case where $n\geq 2$ in the following. Let $z\in D_s$ be a closed point such that $(D_s)_{\mathrm{red}}$ is smooth at $z$ and that the ramification of $(j_!\mathscr F)|_{X_s}$ along $(D_{s})_{\mathrm{red}}$ is non-degenerate at $z$. After replacing $X$ by an open neighborhood of $z$, we can find a smooth $k$-curve $C$ and a closed immersion $\iota:C\rightarrow X_s$ such that
the curve $C$ intersects $(D_{s})_{\mathrm{red}}$ transversally at $z$ and that the immersion $\iota:C\rightarrow X_s$ is $SS(j_!\mathscr F|_{X_s})$-transversal at $z$ (Proposition \ref{transversalcurve}).
By Theorem \ref{Ccutbycurve}(2), we have 
\begin{equation}\label{cdprop1_f1}
\rc_z(j_!\mathscr F|_C)=\rc_{(D_{s})_{\mathrm{red}}}((j_!\mathscr F)|_{X_s}).
\end{equation}
Choose a regular system of parameters $\bar t_1,\cdots, \bar t_n$ of $\mathscr O_{X_s,z}$, such that $(\bar t_2,\cdots,\bar t_{n})$ is the kernel of $\iota^\sharp:\mathscr O_{X_s,z}\rightarrow \mathscr O_{C,z}$  and that $(\bar t_1)$ is the kernel of $\mathscr O_{X_s,z}\rightarrow \mathscr O_{(D_{s})_{\mathrm{red}},z}$. For each $2\leq r\leq n$, choose a lifting $t_r\in \mathscr O_{X,z}$ of $\bar t_r\in \mathscr O_{X_s,z}$. We define an
$\mathscr O_{S,s}$-homomorphism $ g_z:\mathscr O_{S,s}[T_2,\cdots,T_{n}]\rightarrow \mathscr O_{X, z}$ by
\begin{equation*}
g_z:\mathscr O_{S,s}[T_2,\cdots ,T_{n}]\rightarrow \mathscr O_{X, z},\ \ \ T_r\mapsto t_r.
\end{equation*}
After replacing $X$ by a Zariski neighborhood of $z$, the map $g_z$ induces an $S$-morphism $g:X\rightarrow  \mathbb A^{n-1}_S$. It satisfies the following conditions after shrinking $X$ further
\begin{itemize}\itemsep=0.2cm
\item[(i)]
it is smooth and of relative dimension $1$;
\item[(ii)]
the restriction $g|_D:D\to\mathbb A^{n-1}_S$ is quasi-finite and flat and, for each $i\in I$, the restriction $g|_{D_i}:D_i\rightarrow \mathbb A^{n-1}_S$ is \'etale (\cite[I, Chapitre 0, 15.1.16]{EGA4});
\item[(iii)]
the curve $C$ is the pre-image $g^{-1}(s\times O)$, where $s\times O$ denotes the product of $s\in S$ and the origin $O\in\mathbb A^{n-1}_k$.
\end{itemize}
Notice that $z$ is an isolated point of $C\bigcap D$. By \cite[IV, 18.12.1]{EGA4}, we have a connected \'etale neighborhood $\gamma:W\to \mathbb A^{n-1}_S$ of $s\times O\in\mathbb A^{n-1}_S$ such that 
\begin{itemize}\itemsep=0.2cm
\item[(1)]
the pre-image $w=\gamma^{-1}(s\times O)$ is a point;
\item[(2)]
the fiber product $D_{W}=D\times_{\bA^{n-1}_S}W$ is a disjoint union of two schemes $E_{W}$ and $H_{W}$ such that the canonical maps $E_{W}\to W$ is finite and flat, and $z'=g_W^{-1}(w)\bigcap E_W$ is a single point which is the pre-image of $z\in D$ in $D_W$, where $g_W:X_W=X\times_{\bA^{n-1}_S}W\to W$ is the base change of $g:X\to \bA^{n-1}_S$ by $\gamma:W\to \bA^{n-1}_S$;
\item[(3)]
for each $i\in I$, the fiber product $D_{W,i}=D_i\times_{\bA^{n-1}_S}W$ is a disjoint union of $E_{W,i}=D_{W,i}\bigcap E_W$ and $H_{W,i}=D_{W,i}\bigcap H_W$, such that $E_{W,i}\to W$ is an isomorphism.
\end{itemize}
Put $P_W=\bigcup_{i\neq i' (i,i'\in I)} E_{W,i}\cap E_{W,i'}$. Since $z'\in P_W$, $P_W\subset E_W$ has codimension $1$ and $E_W\to W$ is finite and flat, the image $g_W(P_W)$ is a codimensional $1$ closed subset of $W$ containing $w$. Since $\gamma:W\to\mathbb A^{n-1}_S$ is \'etale, the closure $Z=\overline{\gamma(g_W(P_W))}$ contains $s\times O$ and has codimension $1$ in $\mathbb A^{n-1}_S$. By Proposition \ref{goodsection}, we have an open dense subscheme $V$ of $S$ and a morphism $\sigma:V\rightarrow \mathbb A^{n-1}_S$ such that the composition of $\sigma:V\to \bA^{n-1}_S$ and $\pi:\mathbb A^{n-1}_S\rightarrow S$ is the canonical injection, that $S\not\subset Z$ and that $\sigma(s)=s\times O$. Since $S$ is irreducible, $\sigma^{-1}(Z)\subset V$ is a closed subset of $S$ which does not contain the generic point. Since $s\times O\in\gamma(W)$, the fiber product $W\times_{\mathbb A^{n-1}_S,\sigma}V$ is non-empty. Let $T$ be the connected component of $W\times_{\mathbb A^{n-1}_S,\sigma}V$ containing $w$. Let $\eta$ be the generic point of $S$ and $\bar\eta$ an algebraic geometric point above $\eta$ that factors through $T $. We have the following commutative diagram:
\begin{equation*}
\xymatrix{\relax
D\ar[d]&E_W\ar[l]\ar[d]\ar@{}|-(0.5){\Box}[rd]&E\ar[l]\ar[d]\ar@{}|-(0.5){\Box}[rd]&E_{\overline\eta}\ar[l]\ar[d]\\
X\ar[d]_-(0.5)g\ar@{}|-(0.5){\Box}[rd]&X_W\ar[l]\ar[d]^-(0.5){g_W}\ar@{}|-(0.5){\Box}[rd]&Y\ar[l]\ar[d]^-(0.5){h}\ar@{}|-(0.5){\Box}[rd]&Y_{\overline\eta}\ar[l]\ar[d]\\
\mathbb A^{n-1}_S&W\ar[l]_-(0.5)\gamma&T\ar[ld]^-(0.5){\gamma'}\ar[l]_-(0.5){\sigma'}&\overline\eta\ar[l]\\
&V\ar[lu]^-(0.5)\sigma&&}
\end{equation*}
We put $E_{i,\overline\eta}=E_{W,i}\times_W\overline\eta$. Since $E_{W,i}\to W$ is an isomorphism for each $i\in I$, and $\eta\notin \sigma^{-1}(Z)$, we obtain that each $E_{i,\overline\eta}$ is isomorphic to $\overline\eta$ and that $E_{\overline\eta}=\coprod_{i\in I} E_{i,\overline\eta}$.
Applying \cite[Theorem 7.2]{H23} to the sheaf $(j_!\mathscr F)|_Y$ and the relative curve $h:Y\to T$, we get 
\begin{equation}\label{cdprop1_f2}
\sum_{i\in I}\rc_{E_{i,\overline\eta}}((j_!\mathscr F)|_{Y_{\overline\eta}})\geq\rc_z((j_!\mathscr F)|_C).
\end{equation}
Applying Theorem \ref{Ccutbycurve} to the morphism $r:Y_{\overline\eta}\to X_{\overline\eta}=X\times_S\overline\eta$ and the sheaf $(j_!\mathscr F)|_{X_{\overline\eta}}$, we obtain that, for each $i\in I$, 
\begin{equation}\label{cdprop1_f3}
\rc_{D_{i,\eta}}((j_!\mathscr F)|_{X_\eta})\geq \rc_{E_{i,\overline\eta}}((j_!\mathscr F)|_{Y_{\overline\eta}}).
\end{equation}
Combining \eqref{cdprop1_f1}, \eqref{cdprop1_f2}, \eqref{cdprop1_f3}, we obtain that 
\begin{equation}
\sum_{i\in I}\rc_{D_{i,\eta}}((j_!\mathscr F)|_{X_\eta})\geq \rc_{(D_{s})_{\mathrm{red}}}(j_!\mathscr F|_{X_s})
\end{equation}
By the assumptions that each $f|_{D_i}:D_i\to S$ is smooth and that $D_s$ is irreducible, we see that $\sum_{i\in I}\rc_{D_{i,\eta}}((j_!\mathscr F)|_{X_\eta})$ is the coefficient of $\rho^*_s(\rGC_f(j_!\sF))$ and $\rc_{(D_{s})_{\mathrm{red}}}(j_!\mathscr F|_{X_s})$ is the coefficient of $\rC_{X_s}((j_!\mathscr F)|_{X_s})$. Hence, we get the inequality \eqref{cdprop1formula}.
\end{proof}

\begin{remark}
The proof of Proposition \ref{cdprop1} is a mimic of \cite[Proposition 8.2]{HY17}. However, the section $\sigma:V\to \mathbb A^{n-1}_S$ is chosen to be the zero section of $\pi:\mathbb A^{n-1}_S\to S$ in the proof of {\it loc. cit.}. In proof of Proposition \ref{cdprop1}, we fix the flaw.
\end{remark}

\begin{proposition}\label{cdprop2}
There exists an open dense subset $V\subseteq S$ such that, for any point $t\in V$, we have
\begin{equation*}
\rho_t^*(\mathrm{GC}_f(j_!\mathscr F))\leq \mathrm{C}_{X_t}((j_!\mathscr F)|_{X_t}).
\end{equation*}
\end{proposition}
\begin{proof}
This is a Zariski local problem at the generic point of $S$. We may assume that, for each point $s\in S$ and any indices $i,\iota\in I$ ($i\neq \iota$), the fibers $D_{i,s}$ and $D_{\iota,s}$ have distinct irreducible components (Proposition \ref{genetosp2}). Hence the proposition can be reduced to the case where $D$ is irreducible.  
 By Proposition \ref{genetosp1}, there exists a connected smooth $k$-scheme $S'$ and an \'etale map $\gamma:S'\rightarrow S$ such that $D'=D\times_SS'$ is the disjoint union of its irreducible components and that every irreducible component of $D'$ has geometrically irreducible fibers at each point of $S'$. By Lemma \ref{bcofclc}, Lemma \ref{bcofgcglc}(i) and the fact that $\gamma:S'\rightarrow S$ is an open mapping, we can reduce to the case where $S$ is a connected, affine and smooth $k$-scheme and $D_s$ is geometrically irreducible for any $s\in S$. We are left to show that there exists an open neighborhood $V\subset S$ of $\eta$ such that, for any $t\in V$, we have 
 \begin{equation}\label{genericsmallconductor}
 \rc_{D_{\eta}}((j_!\sF)|_{X_\eta})\leq  \rc_{D_{t}}((j_!\sF)|_{X_t}). 
 \end{equation}
 Let $T$ be a connected and smooth $k$-scheme and $\beta:T\ra S$ a flat and generically finite morphism. Notice that, to verify the proposition, it is enough to verify it after base-changing by $\beta:T\to S$.

 We put $n=\dim_k X-\dim_k S$.  When $n=1$, the proposition is due to \cite[Theorem 7.2]{H23}. We consider the case where $n\geq 2$ in the following. Let $\bar\eta\rightarrow S$ be an algebraic geometric point above $\eta$, $\bar z\in D_{\bar\eta}$ a closed point such that $D_{\bar\eta}$ is smooth at $\bar z$ and that the ramification of $\mathscr F|_{U_{\bar\eta}}$ along $D_{\bar\eta}$ is non-degenerate at $\bar z$. Since $S$ can be replaced by a flat and generically finite cover $T$ as above, we may assume that $\overline z$ can be descended to a $k(\eta)$-rational point $z\in D_{\eta}$. Since the function field $k(\eta)$ of $S$ has infinitely many elements, after shrinking $X$, we can find a smooth $k(\eta)$-curve $C$ and a closed immersion $\iota:C\ra X_{\eta}$ such that the curve $C$ intersects $D_{\eta}$ transversally at $z$ and that the base change $\iota_{\overline\eta}:C_{\overline\eta}\rightarrow X_{\overline\eta}$ of $\iota:C\ra X_{\eta}$ is $SS(j_!\mathscr F|_{X_{\overline\eta}})$-transversal at $\overline z$ (proposition \ref{transversalcurve}). By Theorem \ref{Ccutbycurve}(2), we have 
 \begin{equation}\label{genericsmallconductor1}
 \rc_{D_{\eta}}((j_!\mathscr F)|_{X_{\eta}})=\rc_{\overline z}((j_!\mathscr F)|_{C_{\overline\eta}}).
 \end{equation}

Choose a regular system of parameters $ t_1,\cdots, t_n$ of $\mathscr O_{X_{\eta},z}$ such that $(t_2,\cdots, t_{n})$ is the kernel of $\mathscr O_{X_{\eta},z}\rightarrow \mathscr O_{C,z}$
and that $( t_1)$ is the kernel of $ \mathscr O_{X,z}\rightarrow \mathscr O_{D,z}$. We define a $k(\eta)$-morphism $g_{\eta}:k(\eta)[T_2,\cdots T_{n}]\rightarrow \mathscr O_{X_{\eta},z}$ by
\begin{equation*}
g_{\eta}:k(\eta)[T_2,\cdots T_{n}]\rightarrow \mathscr O_{X,z},\ \ \  T_i\mapsto t_i.
\end{equation*}
After shrinking $X$ by a Zariski neighborhood of $z$ again, the map $g_{\eta}$ induces an $S$-morphism $g:X\rightarrow \mathbb A^{n-1}_S$ which satisfies the following conditions
\begin{itemize}\itemsep=0.2cm
\item[(i)]
it is smooth and of relative dimension $1$;
\item[(ii)]
the restriction $g|_D:D\rightarrow \mathbb A^{n-1}_S$ is \'etale;
\item[(iii)]
the curve $C$ is the pre-image $g^{-1}(\eta\times O)$, where $\eta\times O$ denotes the fiber product of $\eta\in S$ and the origin $O\in\mathbb A^{n-1}_k$.
\end{itemize}
Let $\sigma:S\rightarrow \mathbb A^{n-1}_S$ be the zero-section of the canonical projection $\pi:\mathbb A^{n-1}_S\rightarrow S$. We denote by $h:Y\rightarrow S$ the base change of $g:X\rightarrow \mathbb A^{n-1}_S$ by $\sigma:S\rightarrow \mathbb A^{n-1}_S$ and we put $E=Y\times_XD$. Since $h|_E:E\rightarrow S$ is \'etale, there exists a connected \'etale neighborhood $\gamma:S'\ra S$ of the geometric point $\bar\eta\rightarrow S$ such that $E'=E\times_SS'$ is a disjoint union of $E'_1$ and $E'_2$ where $E'_1$ is isomorphic to $S'$ and $E'_1$ contains the pre-image of $z\in E_{\eta}$. We have the following commutative diagram
\begin{equation}
\xymatrix{\relax
D\ar[d]\ar@{}|-(0.5)\Box[rd]&E\ar[d]\ar[l]&E'_1\ar[d]\ar[l]\\
X\ar[d]_g\ar@{}|-(0.5)\Box[rd]&Y\ar[d]^h\ar[l]\ar@{}|-(0.5)\Box[rd]&Y'\ar[d]^{h'}\ar[l]\\
\bA^{n-1}_S&S\ar[l]^-(0.5)\sigma&S'\ar[l]^-(0.5)\gamma}
\end{equation}
For any geometric point $\bar t'\rightarrow S'$, we put $Y'_{\bar t'}=Y'\times_{S'}\bar t'$, put $E'_{1,\bar t'}=E'_1\times_{S'} \bar t'$, put $X_{\bar t'}=X\times_S\bar t'$ and put $D_{\bar t'}=D\times_S\bar t'$.  Notice that $h':Y'\to S'$ is a smooth relative curve and that $h'|_{E'_1}:E'_1\to S'$ is isomorphic. By \cite[Theorem 7.2]{H23}, we can find an open dense subset $V'$ of $S'$ such that, for any algebraic geometric point $\overline t'\to V'$, we have 
\begin{equation}\label{genericsmallconductor2}
\rc_{\overline z}((j_!\mathscr F)|_{C_{\overline\eta}})=\rc_{E'_{1,\bar t'}}((j_!\mathscr F)|_{Y'_{\bar t'}}). 
\end{equation}
Applying Theorem \ref{Ccutbycurve}(1) to the immersion $Y'_{\overline t'}\to X_{\overline t'}$ and the sheaf $(j_!\sF)|_{X_{\overline t'}}$ for any algebraic geometric point $\overline t'\to V'$, we obtain 
\begin{equation}\label{genericsmallconductor3}
\rc_{E'_{1,\bar t'}}((j_!\mathscr F)|_{Y'_{\bar t'}})\leq (\rC_{D_{\overline t'}}((j_!\mathscr F)|_{X_{\overline t'}}), Y'_{\bar t'})_{E'_{1,\bar t'}}=\rc_{D_{\overline t'}}((j_!\mathscr F)|_{X_{\overline t'}})=\rc_{D_{t}}((j_!\mathscr F)|_{X_{t}})
\end{equation}
where $t\in S$ denotes the image of $\overline t'\to S'$ in $S$,  and we put $X_t=X\times_S t$ and $D_t=D\times_S t$.
Combining \eqref{genericsmallconductor1}, \eqref{genericsmallconductor2} and \eqref{genericsmallconductor3}, we obtain
\begin{equation}
 \rc_{D_{\eta}}((j_!\mathscr F)|_{X_{\eta}})\leq \rc_{D_{t}}((j_!\mathscr F)|_{X_{t}})
\end{equation}
for any point $t$ in the open dense subset $V=\gamma(V')$ of $S$. We finish the proof of the proposition.
\end{proof}

\begin{proposition}\label{lcdprop1}
For each closed point $s\in S$, we have 
\begin{equation}\label{lcdprop1formula}
\rho_{s}^*(\mathrm{GLC}_{f}(j_!\mathscr F)+D)\geq\mathrm{LC}_{X_{s}}((j_!\mathscr F)|_{X_{s}})+(D_{ s})_{\mathrm{red}}.
\end{equation}
\end{proposition}
\begin{proof}
We fix a closed point $s\in S$. This is a local statement for the Zariski topology of $X$. After shrinking $X$, we may assume that $X$ is affine and irreducible, that $(D_{s})_{\mathrm{red}}\neq\emptyset$ is irreducible and that $D_s\subseteq D_i$ for each $i\in I$. Notice that, for each $i\in I$, we have  $(D_{i})_s\xrightarrow{\sim} (D_s)_{\mathrm{red}}$. 

We put  $n=\dim_k X-\dim_k S$. When $n=1$, i.e., $f:X\to S$ is a smooth relative curve, the inequality \eqref{lcdprop1formula} is due to \cite[Theorem 7.2]{H23}. We focus on the case where $n\geq 2$.  After replacing $X$ by an open and affine neighborhood of the generic point of $X_s$, we may assume that an element $g_1$ of $\Gamma(X,\sO_X)$ defines $D_1$. Let $\beta$ be an integer co-prime to $p$, 
\begin{equation*}
X'=\mathrm{Spec}(\mathscr O_X[T]/(T^\beta-g_1))
\end{equation*}
a tame cover of $X$ of degree $\beta$ ramified along $D_1$, $\pi':X'\to X$ the canonical projection, $f':X'\to S$ the composition of $\pi':X'\to X$ and $f:X\to S$. For any $s\in S$, we denote by $\rho'_s:X'_s\to X'$ the canonical injection, $h'_s:X'_s\to X_s$ the base change of $h':X'\to X$ by $\rho_s:X_s\to X$. Notice that $f':X'\to S$ is a smooth morphism.
We denote by $D'$ the Cartier divisor $(D\times_XX')_{\mathrm{red}}$ on $X'$, by $D_1'$ the smooth divisor $(T)=(D_1\times_XX')_{\mathrm{red}}$ on $X'$ and by $D'_i$ the Cartier divisor $D_i\times_XX'$ for $i
\in I\backslash\{1\}$. We have $D'=\sum_{i\in I} D'_i$ and $\beta\cdot (D'_1)_s=(D'_i)_s$ for $i\in I\backslash \{1\}$. 

Applying \cite[Proposition 6.3]{H19} to the sheaf $\rho'^*_s h'^*j_!\sF$ on $X'_s$ ramified along the divisor $(D'_1)_s$, we can find a closed point $z'$ of $(D'_1)_s$ with the image $z$ in $(D_1)_s$, an immersion $\iota:C\to X'_s$ from a smooth $k$-curve $C$ satisfying
\begin{itemize}\itemsep=0.2cm
\item[(i)]
the curve $C$ and the smooth divisor $(D'_1)_s$ meet transversally at $z'$ in $X'_s$;
\item[(ii)]
the immersion $\iota:C\to X'_s$ is $SS(\rho'^*_sh'^*j_!\sF)$-transversal at $z'$;
\item[(iii)]
the composition of $\iota:C\to X'_s$ and $h'_s:X'_s\to X_s$ is also an immersion.
\end{itemize}
Applying Theorem \ref{Ccutbycurve} to $\iota:C\to X'_s$ and the sheaf $\rho'^*_sh'^*j_!\sF$, we get
\begin{equation}\label{lcdprop1formula1}
\rc_{z'}((j_!\mathscr F)|_C)=\rc_{(D'_1)_s}(\rho'^*_sh'^*j_!\sF).
\end{equation}
Applying Proposition \ref{logext}, we get
\begin{equation}\label{lcdprop1formula2}
\rlc_{(D'_1)_s}(\rho'^*_sh'^*j_!\sF)=\beta\cdot \rlc_{(D_1)_s}(\rho^*_sj_!\sF).
\end{equation}
Combining \eqref{lcdprop1formula1} and \eqref{lcdprop1formula2}, we have 
\begin{equation}\label{lcdprop1formula3}
\rc_{z'}((j_!\mathscr F)|_C)=\rc_{(D'_1)_s}(\rho'^*_sh'^*j_!\sF)\geq \rlc_{(D'_1)_s}(\rho'^*_sh'^*j_!\sF)=\beta\cdot \rlc_{(D_1)_s}(\rho^*_sj_!\sF).
\end{equation}

Choose a regular system of parameters $\bar t_1,\cdots, \bar t_n$ of $\mathscr O_{X'_s,z'}$, such that $(\bar t_2,\cdots,\bar t_{n})$ is the kernel of $\iota^\sharp:\mathscr O_{X'_s,z'}\rightarrow \mathscr O_{C,z'}$  and that $(\bar t_1)$ is the kernel of $\mathscr O_{X'_s,z'}\rightarrow \mathscr O_{(D'_{1})_{s},z'}$. For each $2\leq r\leq n$, choose a lifting $t_r\in \mathscr O_{X',z'}$ of $\bar t_r\in \mathscr O_{X'_s,z'}$. We define an
$\mathscr O_{S,s}$-homomorphism $ \psi_{z'}:\mathscr O_{S,s}[T_2,\cdots,T_{n}]\rightarrow \mathscr O_{X', z'}$ by
\begin{equation*}
\psi_z:\mathscr O_{S,s}[T_2,\cdots ,T_{n}]\rightarrow \mathscr O_{X', z'},\ \ \ T_r\mapsto t_r.
\end{equation*}
After replacing $X'$ by a Zariski neighborhood of $z'$, the map $\psi_{z'}$ induces an $S$-morphism $\psi:X'\rightarrow  \mathbb A^{n-1}_S$. It satisfies the following conditions after shrinking further $X'$
\begin{itemize}\itemsep=0.2cm
\item[(i)]
it is smooth and of relative dimension $1$;
\item[(ii)]
the restriction $g|_D:D\to\mathbb A^{n-1}_S$ is quasi-finite and flat and, for each $i\in I\backslash\{1\}$, the restriction $\psi|_{D'_i}:D'_i\rightarrow \mathbb A^{n-1}_S$ is quasi-finite and  $\gamma|_{D'_1}:D'_1\rightarrow \mathbb A^{n-1}_S$ is \'etale (\cite[I, Chapitre 0, 15.1.16]{EGA4});
\item[(iii)]
the curve $C$ is the pre-image $\psi^{-1}(s\times O)$, where $s\times O$ denotes the product of $s\in S$ and the origin $O\in\mathbb A^{n-1}_k$.
\end{itemize}
Notice that $z$ is an isolated point of $C\bigcap D'\subset X'$. By \cite[IV, 18.12.1]{EGA4}, we have a connected \'etale neighborhood $\gamma:W\to \mathbb A^{n-1}_S$ of $s\times O\in\mathbb A^{n-1}_S$ such that 
\begin{itemize}\itemsep=0.2cm
\item[(1)]
the pre-image $w=\gamma^{-1}(s\times O)$ is a point;
\item[(2)]
the fiber product $D'_{W}=D'\times_{\bA^{n-1}_S}W$ is a disjoint union of two schemes $E'_{W}$ and $H'_{W}$ such that the canonical maps $E'_{W}\to W$ is finite and flat, and $x'=\psi_W^{-1}(w)\bigcap E_W$ is a single point which is the pre-image of $z'\in D'$ in $D'_W$, where $\psi_W:X'_W=X'\times_{\bA^{n-1}_S}W\to W$ is the base change of $\psi:X'\to \bA^{n-1}_S$ by $\gamma:W\to \bA^{n-1}_S$;
\item[(3)]
for each $i\in I\backslash\{1\}$, the fiber product $D'_{W,i}=D'_i\times_{\bA^{n-1}_S}W$ is a disjoint union of $E'_{W,i}=D'_{W,i}\bigcap E'_W$ and $H'_{W,i}=D'_{W,i}\bigcap H'_W$ such that $E'_{W,i}\to W$ is finite and flat over $W$, and the fiber product $D'_{W,1}=D'_1\times_{\bA^{n-1}_S}W$ is a disjoint union of $E'_{W,1}=D'_{W,1}\bigcap E'_W$ and $H'_{W,1}=D'_{W,1}\bigcap H'_W$ such that $E'_{W,1}\to W$ is an isomorphism.
\end{itemize}
Put $P'_W=\bigcup_{i\neq i' (i,i'\in I)} E'_{W,i}\cap E'_{W,i'}$. Since $x'\in P'_W$, $P'_W\subset E'_W$ has codimension $1$ and $E'_W\to W$ is finite and flat, the image $\psi_W(P'_W)$ is a codimensional $1$ closed subset of $W$ containing $w$. Since $\gamma:W\to\mathbb A^{n-1}_S$ is \'etale, the closure $Z=\overline{\gamma(\psi_W(P'_W))}$ contains $s\times O$ and has codimension $1$ in $\mathbb A^{n-1}_S$. By Proposition \ref{goodsection}, we have an open dense subscheme $V$ of $S$ and a morphism $\sigma:V\rightarrow \mathbb A^{n-1}_S$ such that the composition of $\sigma:V\to \bA^{n-1}_S$ and $\pi:\mathbb A^{n-1}_S\rightarrow S$ is the canonical injection, that $S\not\subset Z$ and that $\sigma(s)=s\times O$. Since $S$ is irreducible, $\sigma^{-1}(Z)\subset V$ is a closed subset of $S$ which does not contain the generic point. Since $s\times O\in\gamma(W)$, the fiber product $W\times_{\mathbb A^{n-1}_S,\sigma}V$ is non-empty. Let $T$ be the connected component of $W\times_{\mathbb A^{n-1}_S,\sigma}V$ containing $w$. Let $\eta$ be the generic point of $S$ and $\bar\eta$ an algebraic geometric point above $\eta$ that factors through $T $. We have the following commutative diagram:
\begin{equation*}
\xymatrix{\relax
D\ar[d]&D'\ar[d]\ar[l]&E'_W\ar[l]\ar[d]\ar@{}|-(0.5){\Box}[rd]&E'\ar[l]\ar[d]\ar@{}|-(0.5){\Box}[rd]&E'_{\overline\eta}\ar[l]\ar[d]\\
X\ar[d]_f&X'\ar[l]_{h'}\ar[d]_-(0.5)\psi\ar@{}|-(0.5){\Box}[rd]&X'_W\ar[l]\ar[d]_-(0.5){\psi_W}\ar@{}|-(0.5){\Box}[rd]&Y'\ar[l]\ar[d]_-(0.5){\psi_T}\ar@{}|-(0.5){\Box}[rd]&Y'_{\overline\eta}\ar[l]\ar[d]\\
S&\mathbb A^{n-1}_S\ar[l]_-(0.5)\pi&W\ar[l]_-(0.5)\gamma&T\ar[ld]^-(0.5){\gamma'}\ar[l]_-(0.5){\sigma'}&\overline\eta\ar[l]\\
&&V\ar[lu]^-(0.5)\sigma&&}
\end{equation*}
We put $(E'_i)_{\overline\eta}=E'_{W,i}\times_W\overline\eta$ and put $(H'_i)_{\overline\eta}=H'_{W,i}\times_W\overline\eta$. Since $\eta\notin \sigma^{-1}(Z)$, We have $E'_{\overline\eta}=\coprod_{i\in I}(E'_i)_{\overline\eta}$. Since $E'_{W,i}\to W$ is finite and flat, we have 
\begin{align}
\mathrm{length}_{\overline\eta}((E'_i)_{\overline\eta})=\mathrm{length}_{k}(E'_{W,i}\times_Ww)=m_{z'}(\gamma^*((D'_i)_s))=\beta,
\end{align}
 for any $i\in I\backslash\{1\}$. Since $E'_{W,1}\to W$ is an isomorphism, we have $(E'_1)_{\overline\eta}\cong\overline\eta$. Applying  \cite[Theorem 7.2]{H23} to the relative curve $\psi:Y'\to T$, and the sheaf $(j_!\mathscr F)|_{Y'}$, we get 
 \begin{equation}\label{lcdprop1formula4}
\sum_{i\in I}\sum_{y\in (E'_i)_{\overline\eta}}\rc_y((j_!\mathscr F)|_{Y'_{\overline\eta}})\geq \rc_{z'}((j_!\mathscr F)|_C).
 \end{equation}
 We put $X_{\overline\eta}=X\times_S\overline\eta$, put $D_{\overline\eta}=D\times_S\overline\eta$, put $(D_i)_{\overline\eta}=D_i\times_S\overline\eta$ for $i\in I$, put $X'_{\overline\eta}=X'\times_S\overline\eta$, put $D'_{\overline\eta}=D'\times_S\overline\eta$ and put $(D'_i)_{\overline\eta}=D'_i\times_S\overline\eta$ for $i\in I$. Notice that $(E'_i)_{\overline\eta}\coprod (H'_i)_{\overline\eta}=(D'_i)_{\overline\eta}\times_{X'_{\overline\eta}}Y'_{\overline\eta}$ for $i\in I$. Applying Theorem \ref{Ccutbycurve} to the closed immersion $Y'_{\overline\eta}\to X'_{\overline\eta}$ and the sheaf $(j_!\mathscr F)|_{X'_{\overline\eta}}$ ramified along the divisor $(D'_1)_{\overline\eta}$, we have
 \begin{align}
 \rc_{(D'_1)_{\eta}}((j_!\mathscr F)|_{X'_{\eta}})&\geq  \rc_{(E'_1)_{\eta}}((j_!\mathscr F)|_{Y'_{\eta}}).\label{lcdprop1formula5}
 \end{align}
 Applying Theorem \ref{Ccutbycurve} to the quasi-finite morphism $Y'_{\overline\eta}\to X_{\overline\eta}$ and the sheaf $(j_!\mathscr F)|_{X_{\overline\eta}}$ along the divisor $(D_i)_{\overline\eta}$ $(i\in I\backslash\{1\})$, we have 
 \begin{equation}\label{lcdprop1formula6}
 \beta\cdot\rc_{(D_i)_{\eta}}((j_!\sF)|_{X_{\overline\eta}})\geq \sum_{y\in (E'_i)_{\overline\eta}}\rc_y((j_!\mathscr F)|_{Y'_{\overline\eta}}).
 \end{equation}
 By \eqref{lcdprop1formula4}, \eqref{lcdprop1formula5} and \eqref{lcdprop1formula6}, we have 
 \begin{equation}\label{lcdprop1formula7}
 \rc_{(D'_1)_{\eta}}((j_!\mathscr F)|_{X'_{\eta}})+\beta\sum_{i\in I\backslash\{1\}}\rc_{(D_i)_{\eta}}((j_!\sF)|_{X_{\overline\eta}})\geq \rc_{z'}((j_!\mathscr F)|_C). 
 \end{equation}
By \eqref{lcdprop1formula3} and \eqref{lcdprop1formula7}, we have 
\begin{equation}\label{lcdprop1formula8}
\rc_{(D'_1)_{\eta}}((j_!\mathscr F)|_{X'_{\eta}})+\beta\sum_{i\in I\backslash\{1\}}\rc_{(D_i)_{\eta}}((j_!\sF)|_{X_{\overline\eta}})\geq \beta\cdot \rlc_{(D_1)_s}(\rho^*_sj_!\sF).
\end{equation}
Since $\pi':X'\to X$ is tamely ramified along the divisor $D_i$ with degree $\beta$, we obtain that
\begin{equation}\label{lcdprop1formula9}
\beta\cdot \rlc_{(D_1)_{\eta}}((j_!\mathscr F)|_{X_{\eta}})+1=\rlc_{(D'_1)_{\eta}}((j_!\mathscr F)|_{X'_{\eta}})+1\geq \rc_{(D'_1)_{\eta}}((j_!\mathscr F)|_{X'_{\eta}}).
\end{equation}
By subsection \ref{introlc+1>c>lc}, we have 
\begin{equation}\label{lcdprop1formula10}
\rlc_{(D_i)_{\eta}}((j_!\mathscr F)|_{X_{\eta}})+1\geq \rc_{(D_i)_{\eta}}((j_!\mathscr F)|_{X_{\eta}}),\ \ \textrm{for}\ \ i\in I\backslash\{1\}.
\end{equation}
By \eqref{lcdprop1formula8}, \eqref{lcdprop1formula9} and \eqref{lcdprop1formula10}, we have 
\begin{equation}\label{lcdprop1formula11}
\beta\sum_{i\in I}\rlc_{(D_i)_{\eta}}((j_!\mathscr F)|_{X_{\eta}})+\beta\cdot(\sharp I-1)+1\geq \beta\cdot \rlc_{(D_1)_s}(\rho^*_sj_!\sF).
\end{equation}
Divide both sides of \eqref{lcdprop1formula11} by $\beta$ and pass $\beta\to +\infty$, we obtain 
\begin{equation}
\sum_{i\in I}(\rlc_{(D_i)_{\eta}}((j_!\mathscr F)|_{X_{\eta}})+1)\geq \rlc_{(D_1)_s}(\rho^*_sj_!\sF)+1, 
\end{equation}
which implies \eqref{lcdprop1formula}.
\end{proof}

\begin{proposition}\label{lcdprop2}
There exists an open dense subset $V$ of $S$, such that for each point $s\in V$, we have 
\begin{equation}
\rho^*_s(\rGLC_f(j_!\sF)+D)=\rLC_{X_s}((j_!\sF)|_{X_s})+(D_{s})_{\mathrm{red}}.
\end{equation}
\end{proposition}
\begin{proof}
By Proposition \ref{genetosp2}, we can find an open dense subset $W\subseteq S$ such that, for any $s\in W$ and any different indices $i,\iota\in I$, the fibers $(D_i)_s$ and $(D_\iota)_s$ do not have same irreducible components. Hence, we may assume that $D$ is irreducible. Therefore, to prove the proposition, it is sufficient to prove the existence of an open dense subset $V$ of $S$ such that, for any point $s\in V$, we have 
\begin{equation}\label{lcdprop2formula0}
(\rlc_{D_{\eta}}(\rho^*_\eta j_!\sF))\cdot D_s=\rLC_{X_s}(\rho^*_sj_!\sF).
\end{equation}
Notice that $D_s$ may not be irreducible in general. By proposition \ref{genetosp1}, there exists a connected smooth $k$-scheme $S'$ and an \'etale map $\gamma:S'\rightarrow S$ such that $D'=D\times_SS'$ is the disjoint union of its irreducible components and that every irreducible component of $D'$ has geometrically irreducible fibers at each point of $S'$. By Lemma \ref{bcofclc}, Lemma \ref{bcofgcglc}(i) and the fact that $\gamma:S'\rightarrow S$ is an open mapping, we can reduce to the case where $S$ is a connected, affine and smooth $k$-scheme and $D_s$ is geometrically irreducible for any $s\in S$. We may further replace $X$ by an affine neighborhood of the generic point of $D$, Thus, we may assume that $X$ and $S$ are connected, affine and smooth and moreover $D$ is defined by an element $g$ of $\Gamma(X,\mathscr O_X)$ with geometrically irreducible fibers. Let $\beta$ be a positive integer co-prime to $p$ with $\beta\geq(\rk_{\Lambda}\sF)^2+1$. Let 
\begin{equation*}
X'=\spec(\sO_X[T]/(T^\beta-g))
\end{equation*}
be a tame cover of $X$ ramified along $D$ of degree $\beta$, $h':X'\to X$ the canonical projection, $f':X'\to S$ the composition of $h':X'\to X$ and $f:X\to S$ and $S'$ the smooth divisor on $X$ defined by $(T)=(X\times_SS')_{\mathrm{red}}$.
For any $s\in S$, we denote by $\rho'_s:X'_s\to X'$ the canonical injection, $h'_s:X'_s\to X_s$ the base change of $h':X'\to X$ by $\rho_s:X_s\to X$. Notice that $f':X'\to S$ and $f'|_{D'}:D'\to S$ are smooth and that $h'^*_s(D_s)=\beta\cdot D'_s$. Applying Theorem \ref{intromaintheorem}(i) to the morphism $f':X'\to S$ and the sheaf $h'^*j_!\sF$, we can find an open dense subset $V$ of $S$ such that, for any $s\in V$, we have 
\begin{equation}\label{lcdprop2formula1}
\rc_{D'_{\eta}}(\rho'^*_{\eta}h'^*j_!\sF)\cdot D'_s=\rC_{X'_{s}}(h'^*_s\rho^*_sj_!\sF).
\end{equation}
In the following, let $s$ be a point of $V$. By subsection \ref{introlc+1>c>lc} and Proposition \ref{logext}, we have 
\begin{align}
(\beta\cdot\rlc_{D_{\eta}}(\rho^*_\eta j_!\sF)+1)\cdot D'_s&=(\rlc_{D'_{\eta}}(\rho'^*_{\eta}h'^*j_!\sF)+1)\cdot D'_s\geq \rc_{D'_{\eta}}(\rho'^*_{\eta}h'^*j_!\sF)\cdot D'_s\label{lcdprop2formula2}\\
&\geq \rlc_{D'_{\eta}}(\rho'^*_{\eta}h'^*j_!\sF)\cdot D'_s=(\beta\cdot\rlc_{D_{\eta}}(\rho^*_\eta j_!\sF))\cdot D'_s,\nonumber
\end{align}
and 
\begin{align}
h'^*_s(\rLC_{X_s}(\rho^*_sj_!\sF))+D'_s&=\rLC_{X'_s}(h'^*_s\rho^*_sj_!\sF)+D'_s\geq \rC_{X'_{s}}(h'^*_s\rho^*_sj_!\sF)\label{lcdprop2formula3}\\
&\geq \rLC_{X'_s}(h'^*_s\rho^*_sj_!\sF)=h'^*_s(\rLC_{X_s}(\rho^*_sj_!\sF)).\nonumber
\end{align}
By \eqref{lcdprop2formula1}, \eqref{lcdprop2formula2} and \eqref{lcdprop2formula3},  we have 
\begin{align}
(\beta\cdot\rlc_{D_{\eta}}(\rho^*_\eta j_!\sF)+1)\cdot D'_s\geq h'^*_s(\rLC_{X_s}(\rho^*_sj_!\sF))\geq (\beta\cdot\rlc_{D_{\eta}}(\rho^*_\eta j_!\sF)-1)\cdot D'_s.
\end{align}
It is equivalent to
\begin{equation}\label{lcdprop2formula4}
\left(\rlc_{D_{\eta}}(\rho^*_\eta j_!\sF)+\frac{1}{\beta}\right)\cdot D_s\geq \rLC_{X_s}(\rho^*_sj_!\sF)\geq \left(\rlc_{D_{\eta}}(\rho^*_\eta j_!\sF)-\frac{1}{\beta}\right)\cdot D_s.
\end{equation}
Let $\rlc_1$ be a coefficient of $\rLC_{X_s}(\rho^*_sj_!\sF)$. By \eqref{lcdprop2formula4}, we have 
\begin{equation}\label{lcdprop2formula5}
|\rlc_1-\rlc_{D_{\eta}}(\rho^*_\eta j_!\sF)|\leq \frac{1}{\beta}\leq\frac{1}{(\rk_{\Lambda}\sF)^2+1}
\end{equation}
By Proposition \ref{Hasse-Arf-conductor}, we can find a positive integer $0<r\leq(\rk_{\Lambda}\sF)^2$ such that 
\begin{equation}\label{lcdprop2formula6}
\rlc_1-\rlc_{D_{\eta}}(\rho^*_\eta j_!\sF)\in \frac{1}{r}\mathbb Z.
\end{equation}
Combining \eqref{lcdprop2formula5} and \eqref{lcdprop2formula6}, we obtain that, for any $s\in V$, the equality \eqref{lcdprop2formula0} holds
which finishes the proof of the proposition.
\end{proof}

\subsection{}\label{finalproofCandLCtheorem}{\bf Proof of Theorem \ref{intromaintheorem}.} 

For a non-closed point $t$ of $S$, we denote by $T$ the smooth part of $\overline{\{t\}}$, which is an open dense subset of $\overline{\{t\}}$. We have the following commutative diagram
\begin{equation}\label{XTXsXdiagramlc}
\xymatrix{\relax
X\ar[d]_f\ar@{}|-{\Box}[rd]&X_T\ar[d]^{f_T}\ar[l]^-(0.5){\rho_T}&X_s\ar[l]^-(0.5){\iota_s}\ar@/_1pc/[ll]_{\rho_s}\\
S&T\ar[l]
}
\end{equation}
where $s\in T$ is a point and $\iota_s:X_s\rightarrow X_T$ the base change of the inclusion $s\rightarrow T$.

\vspace{0.3cm}

We firstly prove Theorem \ref{intromaintheorem}(i). It is divided into the following four steps:

Step 1. Combining Proposition \ref{cdprop1} and Proposition \ref{cdprop2}, we can find an open dense subset $V\subseteq S$ such that 
\begin{equation}\label{eqclosedpoint}
\rho_s^*(\mathrm{GC}_f(j_!\mathscr F))=\mathrm{C}_{X_s}(j_!\mathscr F|_{X_s}).
\end{equation}
for each closed point $s\in V$. 

Step 2.  Let $V$ be the open dense subset of $S$ in Step 1. For a non-closed point $t\in V$, we denote by $T$ the smooth part of $\overline{\{t\}}$. We take the notation of \eqref{XTXsXdiagramlc}. Note that $T\cap V$ is an open dense subset of $T$. By Step 1, for every closed point $s\in T\cap V$, we have
\begin{equation}\label{Cstep2formula1}
\iota^*_s(\rho^*_T(\mathrm{GC}_f(j_!\mathscr F)))=\rho_s^*(\mathrm{GC}_f(j_!\mathscr F))=\mathrm{C}_{X_s}((j_!\mathscr F)|_{X_s}).
\end{equation}
Applying Step 1 to $ f_T:X_T\rightarrow T$ and $(j_!\mathscr F)|_{X_T}$, we can find an open dense subset $W\subset T\cap V$ such that, for every closed point $s\in W$, we have
\begin{equation}\label{Cstep2formula2}
\iota^*_s(\mathrm{GC}_{f_T}((j_!\mathscr F)|_{X_T}))=\mathrm{C}_{X_s}((j_!\mathscr F)|_{X_s}).
\end{equation}
By \eqref{Cstep2formula1} and \eqref{Cstep2formula2}), we have 
\begin{equation*}
\iota^*_s(\rho^*_T(\mathrm{GC}_f(j_!\mathscr F)))=\iota^*_s(\mathrm{GC}_{f_T}((j_!\mathscr F)|_{X_T})),
\end{equation*}
for any closed point $s\in W$. Since both $W$ and $T$ are $k$-schemes and $W$ is dense in $T$, we obtain that $\rho^*_T(\mathrm{GC}_f(j_!\mathscr F))=\mathrm{GC}_{f_T}(j_!\mathscr F|_{X_T})$. Applying $\iota^*_t$ to both sides of the equation, we get
\begin{equation*}
\rho_t^*(\mathrm{GC}_f(j_!\mathscr F))=\mathrm{C}_{X_t}((j_!\mathscr F)|_{X_t}).
\end{equation*}
Combining with Step 1, we prove that
\begin{equation*}
\rho_t^*(\mathrm{GC}_f(j_!\mathscr F))=\mathrm{C}_{X_t}((j_!\mathscr F)|_{X_t}),
\end{equation*}
 for any $t\in V$.

Step 3. Let $V$ be the open dense subset of $S$ in Step 1. Let $t\in S-V$ be a point. If $t$ is closed, we have (Proposition \ref{cdprop1})
\begin{equation}\label{Cstep3formula1}
\rho_t^*(\mathrm{GC}_f(j_!\mathscr F))\geq \mathrm{C}_{X_t}(j_!\mathscr F|_{X_t}).
\end{equation}
When $t$ is not closed, we denote by $T$ the smooth part of $\overline{\{t\}}$ and we take the notation of  \eqref{XTXsXdiagramlc}. Applying Step 1 to $f_T:X_T\rightarrow T$ and $(j_!\mathscr F)|_{X_T}$, there exists an open dense subset $W\subseteq T$ such that, for any closed point $s\in W$, we have
\begin{equation}\label{Cstep3formula2}
\iota^*_s(\mathrm{GC}_{f_T}((j_!\mathscr F)|_{X_T}))=\mathrm{C}_{X_s}((j_!\mathscr F)|_{X_s}).
\end{equation}
By Proposition \ref{cdprop1}, for any closed point $s\in W$, we have
\begin{equation}\label{Cstep3formula3}
\iota^*_s(\rho^*_T(\mathrm{GC}_f(j_!\mathscr F)))=\rho^*_s(\mathrm{GC}_f(j_!\mathscr F))\geq\mathrm{C}_{X_s}((j_!\mathscr F)|_{X_s}).
\end{equation}
  By \eqref{Cstep3formula2}, \eqref{Cstep3formula3}, we get
\begin{equation*}
\iota^*_s(\rho^*_T(\mathrm{GC}_f(j_!\mathscr F)))\geq \iota^*_s(\mathrm{GC}_{f_T}(j_!\mathscr F|_{X_T})),
\end{equation*}
for any closed point $s\in W$. Hence, we have 
\begin{equation}\label{Cstep3formula4}
\rho^*_T(\mathrm{GC}_f(j_!\mathscr F))\geq \mathrm{GC}_{f_T}(j_!\mathscr F|_{X_T})
\end{equation}
Applying $\iota^*_t$ to both sides of \eqref{Cstep3formula3}, we obtain
\begin{equation*}
\rho_t^*(\mathrm{GC}_f(j_!\mathscr F))\geq\mathrm{C}_{X_t}(j_!\mathscr F|_{X_t}).
\end{equation*}
In summary, for any $t\in S-V$, we have 
\begin{equation*}
\rho_t^*(\mathrm{GC}_f(j_!\mathscr F))\geq\mathrm{C}_{X_t}(j_!\mathscr F|_{X_t}).
\end{equation*}

Step 4. We obtain Theorem \ref{intromaintheorem}(i) by the combination of Step 2 and Step 3.

\vspace{0.3cm}

Secondly, we prove Theorem \ref{intromaintheorem}(ii). It consists of the following three steps:

Step I. By Proposition \ref{lcdprop2}, we can find an open dense subset $W\subseteq S$ such that 
\begin{equation}\label{eqclosedpointlc}
\rho_s^*(\mathrm{GLC}_f(j_!\mathscr F))=\mathrm{LC}_{X_s}(j_!\mathscr F|_{X_s}).
\end{equation}
for each point $s\in W$. 

Step II. Let $W$ be the open dense subset of $S$ in Step I. Let $t\in S-W$ be a point. If $t$ is closed, we have (Proposition \ref{lcdprop1})
\begin{equation}\label{LCstep2formula1}
\rho_{t}^*(\mathrm{GLC}_{f}(j_!\mathscr F)+D)\geq\mathrm{LC}_{X_{t}}((j_!\mathscr F)|_{X_{t}})+(D_{t})_{\mathrm{red}}.
\end{equation}
When $t$ is not closed, we denote by $T$ the smooth part of $\overline{\{t\}}$ and we take the notation of  \eqref{XTXsXdiagramlc}. We put $D_T=D\times_ST$.
Applying Step 1 to $f_T:X_T\rightarrow T$ and $(j_!\mathscr F)|_{X_T}$, there exists an open dense subset $V\subseteq T$ such that, for any closed point $s\in V$, we have
\begin{equation}\label{LCstep2formula2}
\iota^*_s(\mathrm{GLC}_{f_T}((j_!\mathscr F)|_{X_T})+(D_T)_{\mathrm{red}})=\mathrm{LC}_{X_s}((j_!\mathscr F)|_{X_s})+(D_s)_{\mathrm{red}}.
\end{equation}
By Proposition \ref{lcdprop1}, for any closed point $s\in V$, we have
\begin{equation}\label{LCstep2formula3}
\iota^*_s(\rho^*_T(\mathrm{GLC}_f(j_!\mathscr F)+D))=\rho^*_s(\mathrm{GLC}_f(j_!\mathscr F)+D)\geq\mathrm{LC}_{X_s}((j_!\mathscr F)|_{X_s})+(D_s)_{\mathrm{red}}.
\end{equation}
  By \eqref{LCstep2formula2}, \eqref{LCstep2formula3}, we get
\begin{equation*}
\iota^*_s(\rho^*_T(\mathrm{GLC}_f(j_!\mathscr F)+D))\geq \iota^*_s(\mathrm{GC}_{f_T}(j_!\mathscr F|_{X_T})+(D_T)_{\mathrm{red}}),
\end{equation*}
for any closed point $s\in V$. Hence, we have 
\begin{equation}\label{LCstep2formula4}
\rho^*_T(\mathrm{GLC}_f(j_!\mathscr F)+D)\geq \mathrm{GLC}_{f_T}(j_!\mathscr F|_{X_T})+(D_T)_{\mathrm{red}}
\end{equation}
Applying $\iota^*_t$ to both sides of \eqref{LCstep2formula4}, we obtain
\begin{align*}
\rho_t^*(\mathrm{GLC}_f(j_!\mathscr F)+D)&=\iota^*_t(\rho^*_T(\mathrm{GLC}_f(j_!\mathscr F)+D))\\
&\geq \iota^*_t(\mathrm{GLC}_{f_T}(j_!\mathscr F|_{X_T})+(D_T)_{\mathrm{red}})\\
&=\mathrm{LC}_{X_t}(j_!\mathscr F|_{X_t})+\iota^*_t((D_T)_{\mathrm{red}})=\mathrm{LC}_{X_t}(j_!\mathscr F|_{X_t})+(D_t)_{\mathrm{red}}.
\end{align*}
In summary, for any $t\in S-W$, we have \begin{equation*}
\rho_t^*(\mathrm{GLC}_f(j_!\mathscr F)+D)\geq\mathrm{LC}_{X_t}(j_!\mathscr F|_{X_t})+(D_t)_{\mathrm{red}}.
\end{equation*}

Step III. We obtain Theorem \ref{intromaintheorem}(ii) by the combination of Step I and Step II.
\hfill$\Box$ 

\begin{corollary}\label{coro_for_Betti_bound}
Assume that $S$ is integral and that the fiber $D_t$ is geometrically integral for each $t\in S$. Then, for any  point $t\in S$, we have 
\begin{equation*}
\rc_D(j_!\mathscr F)\geq\rc_{D_{\eta}}((j_!\mathscr F)|_{X_{\eta}})\geq \rc_{D_{t}}((j_!\mathscr F)|_{X_{ t}})\ \ \ \text{and}\ \ \ \rlc_D(j_!\mathscr F)\geq \rlc_{D_{\eta}}((j_!\mathscr F)|_{X_{\eta}})\geq\rlc_{D_{t}}((j_!\mathscr F)|_{X_{t}}).
\end{equation*}
\end{corollary}

In the corollary, the invariants $\rc_D(j_!\mathscr F)$ and $\rlc_D(j_!\mathscr F)$ are well defined since $X$ and $D$ are smooth in a Zariski neighborhood of the generic point of $D$.

\begin{proof}
By Theorem \ref{intromaintheorem}, there exists an open dense subset $V$ of $S$ such that, for any point $s\in V$, we have 
\begin{equation}\label{Ctheoremcoro1}
\rc_{D_{\eta}}((j_!\mathscr F)|_{X_{\eta}})=\rc_{D_{s}}((j_!\mathscr F)|_{X_{s}})\ \ \ \text{and}\ \ \ \rlc_{D_{\eta}}((j_!\mathscr F)|_{X_{\eta}})=\rlc_{D_{s}}((j_!\mathscr F)|_{X_{s}}),
\end{equation}
 and that, for any point $t\in S$, we have
\begin{equation}\label{Ctheoremcoro2}
\rc_{D_{\eta}}((j_!\mathscr F)|_{X_{\eta}})\geq \rc_{D_{t}}((j_!\mathscr F)|_{X_{t}})\ \ \ \textrm{and}\ \ \ \rlc_{D_{\eta}}((j_!\mathscr F)|_{X_{\eta}})\geq \rlc_{D_{t}}((j_!\mathscr F)|_{X_{t}}).
\end{equation}
 Let $V_0$ be the smooth locus of $V$, which is open dense in $V$. 
 Applying Theorem \ref{Ccutbycurve} and Theorem \ref{LCcutbycurve} to the closed immersion $\rho_v:X_v\to X$ and the sheaf $j_!\mathscr F$ for a closed point $v\in V_0$, we have 
 \begin{equation} \label{Ctheoremcoro3}
 \rc_D(j_!\mathscr F)\geq \rc_{D_{v}}((j_!\mathscr F)|_{X_{v}})\ \ \ \textrm{and} \ \ \ \rlc_D(j_!\mathscr F)\geq \rlc_{D_{v}}((j_!\mathscr F)|_{X_{v}}).
 \end{equation}
Combining \eqref{Ctheoremcoro1}, \eqref{Ctheoremcoro2} and \eqref{Ctheoremcoro3}, we obtain the corollary.
 \end{proof}

\end{document}